\PassOptionsToPackage{table}{xcolor} 
\documentclass[final,leqno,onefignum,onetabnum]{siamart1116}
\usepackage{amsmath,amsfonts,amscd,amssymb}
\usepackage{mathtools} 
\usepackage{algorithm}
\usepackage{enumerate}
\usepackage[noend]{algorithmic}
\usepackage{subcaption}
\usepackage[bottom,hang,flushmargin]{footmisc}

\usepackage{collcell}
\usepackage{graphicx}
\usepackage{tikz}
    \usetikzlibrary{matrix}
    \usetikzlibrary{trees}
    \usetikzlibrary{calc}
    \usetikzlibrary{decorations.markings}
    \usetikzlibrary{intersections}
    \usetikzlibrary{patterns}
\usepackage{flafter} 
\usepackage{boxedminipage}

\usetikzlibrary{plotmarks}
 \usepackage{pgfplots}
  \pgfplotsset{compat=newest}
  \usetikzlibrary{plotmarks}
  \usepackage{grffile}
\usepgfplotslibrary{patchplots,colormaps}

\newcommand{\op}[1]{{\operatorname{#1}}}
\newcommand{\tr}{\operatorname{tr}}
\newcommand{\iprod}[2]{\langle #1 , #2 \rangle}
\newcommand{\proj}{\Pi}
\newcommand{\prox}[2]{\operatorname{prox}_{#1}(#2)}
\newcommand{\proxb}[2]{P_{#1}(#2)}
\newcommand{\argmin}{\mathop{\mathrm{arg\,min}{}}}
\newcommand{\argmax}{\mathop{\mathrm{arg\,max}{}}}

\newcommand{\norm}[1]{\|#1\|}
\newcommand{\abs}[1]{|#1|}
\newcommand{\nor}{\norm{\cdot}}

\newcommand{\bd}[1]{\mathbf{#1}}
\newcommand{\x}{\mathbf{x}}
\newcommand{\y}{\mathbf{y}}
\newcommand{\z}{\mathbf{z}}
\newcommand{\g}{\mathbf{g}}
\newcommand{\ba}{\mathbf{a}}

\newcommand{\one}{\mathbf{1}}
\newcommand{\zero}{0}
\newcommand{\eye}{I}
\renewcommand{\cal}[1]{\mathcal{#1}}
\newcommand{\R}{\mathbb{R}}
\newcommand{\Sbb}{{\mathbb{S}}}
\newcommand{\st}{*}
\newcommand{\eps}{\varepsilon}

\usepackage[mathscr]{euscript}
\newcommand{\Mset}{\mathscr{M}}		
\newcommand{\Mmat}{M}				
\newcommand{\Mbar}{\overline{M}}		
\newcommand{\comp}[1]{\widetilde{#1}}
\newcommand{\Mcomp}{\comp{\Mmat}}	
\newcommand{\Meff}{\Mset_\mathrm{eff}}
\newcommand{\Loss}{\cal{L}}
\newcommand{\Lhat}{\hat\Loss}

\title{Variational Gram Functions: Convex Analysis and Optimization\thanks{An earlier version of this work has appeared as Chapter 3 in \cite{jalali2016convex}.}}
\author{
Amin Jalali\thanks{
Optimization Theme, Wisconsin Institute for Discovery, Madison, WI (\email{amin.jalali@wisc.edu})
}
\and 
Maryam Fazel\thanks{
Department of Electrical Engineering, University of Washington, Seattle, WA 
(\mbox{\email{mfazel@uw.edu}})}
\and 
Lin Xiao\thanks{
Machine Learning Group, Microsoft Research, Redmond, WA 
(\email{lin.xiao@microsoft.com})}
}

\begin{document}
\maketitle
 
\begin{abstract}
We propose a new class of convex penalty functions, called \emph{variational Gram functions} (VGFs), that can promote pairwise relations, such as orthogonality, among a set of vectors in a vector space. These functions can serve as regularizers in convex optimization problems arising from hierarchical classification, multitask learning, and estimating vectors with disjoint supports, among other applications. We study convexity for VGFs, and give efficient characterizations for their convex conjugates, subdifferentials, and proximal operators. We discuss efficient optimization algorithms for regularized loss minimization problems where the loss admits a common, yet simple, variational representation and the regularizer is a VGF. These algorithms enjoy a simple kernel trick, an efficient line search, as well as computational advantages over first order methods based on the subdifferential or proximal maps. We also establish a general representer theorem for such learning problems. Lastly, numerical experiments on a hierarchical classification problem are presented to demonstrate the effectiveness of VGFs and the associated optimization algorithms. 
\end{abstract}

\pagestyle{myheadings}
\thispagestyle{plain}
\markboth{JALALI, FAZEL, XIAO}{VARIATIONAL GRAM FUNCTIONS}

\section{Introduction}		\label{sec:intro}
Let $\x_1, \ldots, \x_m$ be vectors in $\R^n$. It is well known that their pairwise inner products $\x_i^T \x_j\,$, for $i, j=1,\ldots,m\,$, reveal essential information about their relative orientations, and can serve as a measure for various properties such as orthogonality. In this paper, we consider a class of functions that selectively aggregate the pairwise inner products in a variational form,
\begin{equation}
\label{eqn:vec-div-func}
\Omega_{\Mset}(\x_1,\ldots,\x_m) 
= \max_{\Mmat\in\Mset}\; \textstyle\sum_{i,j=1}^m \Mmat_{ij} \x_i^T \x_j \,,
\end{equation}
where $\Mset$ is a compact subset of the set of~$m$ by~$m$ symmetric matrices. 
Let $X=[\x_1 ~\cdots~ \x_m]$ be an $n\times m$ matrix. 
Then the pairwise inner products $\x_i^T \x_j$ are the entries of the 
Gram matrix $X^T X$ and the function above can be written as
\begin{equation}	\label{def:omega}
  \Omega_{\Mset}(X) = \max_{\Mmat\in\Mset} \;\iprod{X^T X}{\Mmat}
  = \max_{\Mmat\in\Mset}\;\tr (X \Mmat X^T)\,,
\end{equation}
where $\iprod{A}{B}=\tr(A^T B)$ denotes the matrix inner product. 
We call $\Omega_{\Mset}$ a \emph{variational Gram function} (VGF) of the 
vectors $\x_1,\ldots,\x_m$ induced by the set~$\Mset$.
If the set~$\Mset$ is clear from the context, we may write $\Omega(X)$ 
to simplify notation.

As an example, consider the case where~$\Mset$ is given by a box constraint,
\begin{align}	\label{eq:Mbox}
  \Mset = \bigl\{ \Mmat :\; |\Mmat_{ij}| \leq \Mbar_{ij},~i, j=1,\ldots,m \bigr\} ,
\end{align}
where $\Mbar$ is a symmetric nonnegative matrix. 
In this case, the maximization in the definition of $\Omega_\Mset$ picks either $\Mmat_{ij}=\Mbar_{ij}$ or $\Mmat_{ij}=-\Mbar_{ij}$ depending on the sign of $\x_i^T\x_j\,$, for all $i,j=1,\ldots,m\,$ (if $\x_i^T\x_j=0$, the choice is arbitrary). Therefore, 
\begin{equation}\label{eq:def-L1Gram}
  \Omega_\Mset(X) 
  = \max_{\Mmat\in\Mset}\; \textstyle\sum_{i,j=1}^m \Mmat_{ij} \x_i^T \x_j 
  = \textstyle\sum_{i,j=1}^m \Mbar_{ij} |\x_i^T \x_j| \,. 
\end{equation}
Equivalently, 
$\Omega_\Mset(X)$ is the weighted sum of the absolute values of pairwise inner products. 
This function was proposed in \cite{ZhouXiaoWu11} as a regularization function to promote orthogonality between selected pairs of linear classifiers in the context of hierarchical classification.

Observe that the function $\tr(X \Mmat X^T)$ is a convex quadratic function of~$X$ if~$\Mmat$ is positive semidefinite.
As a result, the variational form $\Omega_{\Mset}(X)$ is convex if $\Mset$ 
is a subset of the positive semidefinite cone $\Sbb_+^m$,
because then it is the pointwise maximum of a family of convex functions indexed by $\Mmat\in\Mset$ 
(see, e.g., \cite[Theorem~5.5]{Rockafellar70book}). 
However, this is not a necessary condition.
For example, the set~$\Mset$ in~\eqref{eq:Mbox} is not a subset of $\Sbb_+^m$
unless $\Mbar=0$, but the VGF in~\eqref{eq:def-L1Gram} is convex provided that
the \emph{comparison matrix} of $\Mbar$ (derived by negating the off-diagonal entries)
is positive semidefinite \cite{ZhouXiaoWu11}. 
In this paper, we study conditions under which different classes of VGFs are convex and provide unified characterizations for the subdifferential, convex conjugate, and the associated proximal operator for any convex VGF. 
Interestingly, a convex VGF defines a semi-norm\footnote{a semi-norm satisfies all the properties of a norm except 
that it can be zero for a nonzero input.} as 
\begin{equation} \label{def:div-norm}
\|X\|_\Mset := \sqrt{\Omega_\Mset(X)} 
  = \max_{\Mmat\in\Mset}\; \bigl(\textstyle\sum_{i,j=1}^m \Mmat_{ij} \x_i^T \x_j \bigr)^{1/2} \,.
\end{equation}
If $\Mset\subset\Sbb_+^m$, then $\|X\|_\Mset$ is the pointwise maximum of 
the semi-norms $\|{X \Mmat^{1/2}}\|_F$ over all $M\in\Mset$. 

VGFs and the associated norms can serve as penalties or regularization functions in optimization problems 
to promote certain pairwise properties among a set of vector variables
(such as orthogonality in the above example).
In this paper, we consider optimization problems of the form
\begin{align} \label{eq:loss+omega}
  \underset{X\in\R^{n\times m}}{\operatorname{minimize}} \quad \Loss(X) + \lambda\,\Omega_{\Mset}(X)\,,
\end{align}
where $\Loss(X)$ is a convex loss function of the variable 
$X=[\x_1~\cdots~\x_m]$, $\Omega(X)$ is a convex VGF,
and $\lambda>0$ is a parameter to trade off the relative 
importance of these two functions.
We will focus on problems where $\Loss(X)$ is smooth or has an explicit 
variational structure, 
and show how to exploit the structure of $\Loss(X)$ and $\Omega(X)$
together to derive efficient optimization algorithms. More specifically, we employ a unified variational representation for many common loss functions, as
\begin{equation}
\Loss(X) = \max_{\g\in\cal{G}}\;\; \iprod{X}{\cal{D}(\g)} - \Lhat(\g) \,,
\end{equation}
where $\Lhat :\R^p\to\R$ is a convex function, 
$\cal{G}$ is a convex and compact subset of~$\R^p$, 
and $\cal{D}:\R^p\to\R^{n\times m}$ is a linear operator. Exploiting the variational structure in both the loss function and the regularizer allows us to employ efficient primal-dual algorithms, such as mirror-prox \cite{nemirovski2004prox}, which now only require projections onto $\Mset$ and $\mathcal{G}$, instead of computing subgradients or proximal mappings for the loss and the regularizer.

Unfolding the structure for loss functions and regularizers as above, allows us to provide 
a simple preprocessing step for dimensionality reduction, presented in Section \ref{sec:reduced-form}, which can substantially reduce the per iteration cost of any optimization algorithm for \eqref{eq:loss+omega}. 
As another byproduct of these structures, we also present a general representer theorem for problems of the form \eqref{eq:loss+omega} in Section \ref{sec:repr-thm} where the optimal solution is characterized in terms of the input data in a simple and interpretable way.

\paragraph{Organization}
In Section~\ref{sec:examples}, we give more examples of VGFs and explain the
connections with functions of Euclidean distance matrices and 
robust optimization. 
Section~\ref{sec:convexity} studies the convexity of VGFs, as well as their conjugates,
semidefinite representability, corresponding norms and subdifferentials. 
Their proximal operators are derived in Section~\ref{sec:prox}.
In Section~\ref{sec:opt-algm}, we study a class of structured loss minimization
problems with VGF penalities, and show how to exploit their structure, to get an efficient optimization algorithm using a variant of the mirror-prox algorithm with adaptive line search, to use a simple preprocessing step to reduce the computations in each iteration, and to provide a characterization of the optimal solution as a representer theorem. Finally, in Section~\ref{sec:hier-class}, we present a numerical experiment on hierarchical classification to illustrate the application of VGFs.

\paragraph{Notation} 
In this paper, $\Sbb^m$ denotes the set of symmetric matrices in $\R^{m\times m}$,
and $\Sbb_+^m \subset \Sbb^m$ is the cone of positive semidefinite (PSD) matrices. 
We may omit the superscript $m$ when the dimension is clear from the context. 
The symbol $\preceq$ represents the Loewner partial order 
and $\iprod{\cdot}{\cdot}$ denotes the inner product. 
We use capital letters for matrices and bold lower case letters for vectors.
We use $X\in\R^{n\times m}$ and $\x = \text{vec}(X)\in\R^{nm}$ interchangeably,
with $\x_i$ denoting the $i$th column of $X$; 
i.e., $X=[\x_1~\cdots~\x_m]$. 
$\one$ and $\zero$ denote matrices or vectors of all ones and all  zeros respectively, whose sizes would be clear from the context.
The entry-wise absolute value of $X$ is denoted by $\abs{X}\,$. 
$\nor_p$ denotes the $\ell_p$ norm of the input vector or matrix, and $\nor_F$ denotes the Frobenius norm (similar to $\ell_2$ vector norm). 
The convex conjugate of a function~$f$ is defined as 
$f^\st(y) = \sup_{y}\; \iprod{x}{y}-f(x)\,$, and the dual norm of 
$\nor\,$ is defined as $\norm{\y}^\st = \sup\{\iprod{\x}{\y}:\; \norm{\x}\leq 1\}\,$. 
$\argmin$ ($\argmax$) returns an optimal point to a minimization (maximization) program 
while $\op{Arg\,min}$ (or $\op{Arg\,max}$) is the set of all optimal points. 
The operator $\op{diag}(\cdot)$ is 
used to put a vector on the diagonal of a zero matrix of corresponding size, to extract the diagonal entries of a matrix as a vector, or for zeroing out the off-diagonal entries of a matrix. 
We use $f\equiv g$ to denote $f(x) = g(x)$ for all $x\in\op{dom}(f) = \op{dom}(g)\,$.

\section{Examples and connections}	\label{sec:examples}
In this section, we present examples of VGFs associated to different choices of the set $\Mset$. 
The list includes some well known functions that can be expressed in the 
variational form of \eqref{eqn:vec-div-func}, as well as some new ones.
 
\paragraph{Vector norms} 
Any vector norm $\nor$ on $\R^m$ is the square root of a VGF defined by $\Mset = \{\bd{u}\bd{u}^T:\; \norm{\bd{u}}^\st\leq 1 \}$. For a column vector $\x\in\R^m\,$, the VGF is given by 
\[
\Omega_\Mset(\x^T) 
= \max_{\bd{u}}\; \{\tr(\x^T\bd{u}\bd{u}^T\x):\;  \norm{\bd{u}}^\st\leq 1 \,\} 
= \max_{\bd{u}}\; \{(\x^T\bd{u})^2:\;  \norm{\bd{u}}^\st\leq 1 \}  
= \norm{\x}^2\,.
\]

As another example for when $n=1$, consider the case where $\Mset$ is a compact convex set 
of diagonal matrices with positive diagonal entries. 
The corresponding VGF (and norm) is defined as 
\begin{equation}\label{eq:H-norm}
{\Omega_\Mset(\x^T)}= \max_{\theta\in\op{diag}(\Mset)}\, \textstyle\sum_{i=1}^m \theta_i {x_i^2} =  \norm{\x}_\Mset^2  ,
\end{equation}
and the dual norm can be expressed as $(\norm{\x}^\st)^2 = \inf_{\theta\in\op{diag}(\Mset)}\, \sum_{i=1}^m \frac{1}{\theta_i} {x_i^2}  \,$. 
This norm and its dual were first introduced in \cite{micchelli2013regularizers}, in the context of regularization for structured sparsity, and later discussed in \cite{bach2012optimization}. 
The $k$-support norm \cite{argyriou2012sparse}, which is a norm used to encourage vectors to have $k$ or fewer nonzero entries, is
a special case of the dual norm given above, corresponding to $\Mset = \{\op{diag}(\theta):\; 0\leq \theta_i\leq 1\,,\; \one^T\theta = k\}\,$. 

\paragraph{Norms of the Gram matrix}
Given a symmetric nonnegative matrix $\Mbar$, we can define a class of VGFs based on any norm $\|\cdot\|$ and its dual norm $\|\cdot\|_*$. Consider
\begin{equation}\label{eqn:Mset-norm}
  \Mset = \{ K\circ \Mbar :\; \|K\|^\st\leq 1, ~K^T=K\} ,
\end{equation}
where $\circ$ denotes the matrix Hadamard product, 
$(K\circ\Mbar)_{ij}=K_{ij}\Mbar_{ij}$ for all $i,j$.
Then,
\begin{align*}	
  \Omega_\Mset(X) 
  = \max_{\|K\|^\st\leq 1}\; \langle K\circ\Mbar, X^T X \rangle 
  = \max_{\|K\|^\st\leq 1}\; \langle K, \Mbar\circ( X^T X) \rangle
  = \|\Mbar\circ(X^T X)\| \,.
\end{align*}
The followings are several concrete examples.
\begin{romannum}
  \item If we let $\|\cdot\|^*$ in~\eqref{eqn:Mset-norm} be the $\ell_\infty$
    norm, then $\Mset = \{M:~ |M_{ij}/\Mbar_{ij}|\leq 1,\; i,j=1,\ldots,m\}$,
    which is the same as in~\eqref{eq:Mbox}. 
        Here we use the convention $0/0=0$, thus $M_{ij}=0$ whenever 
    $\Mbar_{ij}=0$. 
    In this case, we obtain the VGF in~\eqref{eq:def-L1Gram}:
    \[
      \Omega_\Mset(X) = \|\Mbar\circ(X^T X)\|_1 = \textstyle\sum_{i,j=1}^m \Mbar_{ij} |\x_i^T\x_j|
    \]
  \item 
    If we use the $\ell_2$ norm in~\eqref{eqn:Mset-norm}, then
    $\Mset=\bigl\{M :~ \sum_{i,j}^m (M_{ij}/\Mbar_{ij})^2\leq 1 \bigr\}$ and 
\begin{equation} \label{eq:weighted-Fro}
      \Omega_\Mset(X) = \|\Mbar\circ(X^T X)\|_F
      = \bigl( \textstyle\sum_{i,j=1}^m (\Mbar_{ij} \x_i^T \x_j)^2 \bigr)^{1/2}.
\end{equation}
This function has been considered in multi-task learning \cite{romera2012exploiting}, and also in the context of super-saturated designs \cite{booth1962some,cheng1997s}. 
  \item 
   Using $\ell_1$ norm in~\eqref{eqn:Mset-norm} gives $\Mset = \bigl\{M :~ \sum_{i,j}^m |M_{ij}/\Mbar_{ij}|\leq 1 \bigr\}$ and
    \begin{equation}\label{eq:weighted-Linfty}
      \Omega_\Mset(X) = \|\Mbar\circ(X^T X)\|_\infty 
      = \max_{i,j=1,\ldots,m} \Mbar_{ij} |\x_i^T \x_j| \, .
    \end{equation}
    This case can also be traced back to \cite{booth1962some} in the statistics
    literature, where the maximum of $|\x_i^T \x_j|$ for $i\neq j$
    is used as the measure to choose among supersaturated designs.
\end{romannum}

Many other interesting examples can be constructed this way.
For example, one can model \emph{sharing} vs \emph{competition} using group-$\ell_1$ norm of the Gram matrix which was considered in vision tasks \cite{jayaramandecorrelating}.
We will revisit the above examples to discuss their convexity conditions in Section~\ref{sec:convexity}.

\paragraph{Spectral functions} 
From the definition, the value of a VGF is invariant under
left-multiplication of $X$ by an orthogonal matrix, but this is not true
for right multiplication.
Hence, VGFs are {\em not} functions of singular values 
(e.g., see \cite{lewis1995convex}) in general,
and are functions of the row space of $X$ as well. 
This also implies that in general $\Omega(X) \not\equiv \Omega(X^T)$. 
However, if the set $\Mset$ is closed under left and right
multiplication by orthogonal matrices,
then $\Omega_\Mset(X)$ becomes a function of squared singular values of~$X$. 
For any matrix $\Mmat\in\Sbb^m\,$, denote the sorted vector of its
singular values by $\sigma(\Mmat)$ and let 
$\Theta = \{\sigma(\Mmat):\;\Mmat\in\Mset \}$. Then we have 
\begin{align}	\label{eq:Omega-spectral}
\Omega_\Mset (X) = \max_{\Mmat\in\Mset}\;\tr(X\Mmat X^T) 
= \max_{\theta\in\Theta}\, \textstyle\sum_{i=1}^{\min(n,m)} \theta_i\, \sigma_i(X)^2 \,,
\end{align}
as a result of Von Neumann's trace inequality \cite{mirsky1975trace}. 
Note the similarity of the above to the VGF in \eqref{eq:H-norm}.  
As an example, consider 
\begin{align}	\label{eq:M-sliced-box-spectral}
\Mset=\{\Mmat :\; \alpha_1 \eye \preceq \Mmat \preceq \alpha_2 \eye, ~\tr(\Mmat)=\alpha_3\},
\end{align}
where $0< \alpha_1<\alpha_2$ and $\alpha_3\in[m\alpha_1,~m\alpha_2]$ are given constants. 
The so called \emph{spectral box-norm} \cite{mcdonald2014new} is the dual to the norm of the form~\eqref{def:div-norm} defined via this $\Mset\,$. 
{Note that in this case, $\Mset\subset\Sbb_+^m\,$, so it is easy to see that $\Omega_\Mset$ is convex.}
The square of this norm has been considered in \cite{jacob2008clustered} for clustered multitask learning where it is presented as a convex relaxation for $k$-means. 

\paragraph{Finite set $\Mset\,$}  
For a finite set $\Mset=\{\Mmat_1,\ldots,\Mmat_p\}\subset\Sbb_+^m\,$, the VGF is given by
\[
\Omega_\Mset(X) = \max_{i=1,\ldots,p}\; \norm{X\Mmat_i^{1/2}}_F^2 \,,
\]
i.e., 
the pointwise maximum of a finite number of squared weighted Frobenius norms.

In the following subsections, we consider classes of VGFs which can be used in promoting diversity, have connections to Euclidean distance matrices, or can be interpreted under a robust optimization framework.

\subsection{Diversification}	\label{sec:diversity}
Certain VGFs can be used for {\em diversifying} certain pairs of columns of the input matrix; e.g., minimizing \eqref{eq:def-L1Gram} 
pushes to zero the inner products $\x_i^T\x_j$ corresponding to the nonzero entries in $\Mbar$ as much as possible. 
As another example, observe that two non-negative vectors have disjoint supports if and only if they are orthogonal to each other. Hence, using a VGF as \eqref{eq:def-L1Gram}, $  \Omega_\Mset(X) = \sum_{i,j=1}^m \Mbar_{ij} |\x_i^T \x_j|\,$, that promotes orthogonality, we can define
\begin{align}	\label{eq:def-Omega-abs-X}
\Psi(X) = \Omega_\Mset(\abs{X}) 
\end{align}
to promote disjoint supports among certain columns of $X$; hence diversifying the supports of columns of $X$.  
Convexity of \eqref{eq:def-Omega-abs-X} is discussed in Section \ref{sec:VGF+abs}. 
Different approaches has been used in machine learning applications for promoting diversity; e.g., see\cite{malkin2008ratio,kulesza2012determinantal,iyer2015submodular} and references therein.

\subsection{Functions of Euclidean distance matrix}
Consider a set $\Mset\subset\Sbb^m$ with the property that $\Mmat\one=\zero$
for all $\Mmat\in\Mset$. For every $\Mmat\in\Mset$, let $A =\op{diag}(\Mmat) - \Mmat$ and observe that
\[
\tr(X \Mmat X^T) = \textstyle\sum_{i,j=1}^m \Mmat_{ij}\x_i^T\x_j = \tfrac{1}{2}\textstyle\sum_{i,j=1}^m A_{ij} \norm{\x_i-\x_j}_2^2 \,.
\]
This allows us to express the associated VGF as a function of the 
{\em Euclidean distance matrix $D$}, which is defined by 
$D_{ij} = \tfrac{1}{2} \norm{\x_i-\x_j}_2^2\,$ for $i,j=1,\ldots,m$
(see, e.g., \cite[Section~8.3]{boyd2004convex}).
Let $\mathcal{A} = \{\op{diag}(\Mmat) - \Mmat :\; \Mmat\in\Mset \}\,$. 
Then we have
\[
\Omega_\Mset(X) = \max_{\Mmat\in\Mset}\; \tr(X\Mmat X^T) = \max_{A\in\mathcal{A}}\; \iprod{A}{D} \,.
\]
A sufficient condition for the above function to be convex in $X$ is that 
each $A\in\cal{A}\,$ is entrywise nonnegative, which implies that
the corresponding $\Mmat=\op{diag}(A\one)-A$ is diagonally dominant with nonnegative diagonal elements, hence positive semidefinite. However, this is not a necessary condition and $\Omega_\Mset$ can be convex without all $A$'s being entrywise nonnegative. 

\subsection{Connection with robust optimization}
The VGF-regularized loss minimization problem has the following connection to robust optimization
(see, e.g., \cite{BentalElGhaouiNemirovski2009}): 
the optimization program 
\[
\underset{X}{\operatorname{minimize}}\;\,\max_{\Mmat\in\Mset} \;  \Loss(X) +  \tr(X \Mmat X^T) 
\]
can be interpreted as seeking an $X$ with minimal worst-case value over
an uncertainty set~$\Mset$.
Alternatively, when $\Mset\subset \Sbb^m_+\,$, this can be viewed as a problem with Tikhonov regularization
$\norm{X \Mmat^{1/2}}_F^2$ where the weight matrix $\Mmat^{1/2}$
is subject to errors characterized by the set $\Mset\,$.

We close this section by pointing out that VGFs are different from Quadratic Support Functions introduced in \cite{Aravkin2013:QS}. A closer notion to a VGF is the support functionals studied independently in \cite{burke2015matrix} which correspond to VGFs associated to affine sets $\Mset$, while also allowing for inhomogeneous quadratic functions in their definition.

\section{Convex analysis of VGF}	\label{sec:convexity}

In this section, we study the convexity of VGFs, their conjugate functions and subdifferentials, as well as the related norms.  

First, we review some basic properties. Notice that 
$\Omega_\Mset$ 
is the \emph{support function} of the set $\Mset$ at the Gram matrix $X^TX$; i.e.,
\begin{align}\label{eqn:supp-function}
\Omega_\Mset(X) = \max_{\Mmat\in\Mset}\, \tr(X\Mmat X^T) =S_\Mset(X^TX) 
\end{align}
where the support function of a set $\Mset$ is defined as $S_\Mset(Y) = \sup_{\Mmat\in\Mset}\, \iprod{\Mmat}{Y}\,$
(see, e.g., \cite[Section~13]{Rockafellar70book}).
By properties of the support function (see \cite[Section~15]{Rockafellar70book}), 
\begin{align*}
\Omega_\Mset \equiv \Omega_{\op{conv}(\Mset)} \,,
\end{align*}
where $\op{conv}(\Mset)$ denotes the convex hull of $\Mset\,$.
It is clear that the representation of a VGF (i.e., the associated set $\Mset$) is not unique. 
Henceforth, without loss of generality we assume $\Mset$ is convex 
unless explicitly noted otherwise. 
Also, for simplicity we assume $\Mset$ is a compact set, while all we need is that the maximum in \eqref{eqn:vec-div-func} is attained. For example, a non-compact $\Mset$ that is unbounded along any negative semidefinite direction is allowed. Lastly, we assume $0 \in \Mset$. 

Moreover, VGFs are left unitarily invariant; for any $Y\in \R^{n\times m}$ and any orthogonal matrix $U\in \R^{n\times n}$, where $UU^T=U^TU=\eye$, we have $\Omega(Y) = \Omega(UY)$ and $\Omega^*(Y) = \Omega^*(UY)$; use \eqref{def:omega} and \eqref{eq:conj-pseudo}. We use this property in simplifying computations involving VGFs (such as proximal mapping calculations in Section \ref{sec:prox}) as well as in establishing a general kernel trick and representer theorem in Section \ref{sec:reduced-form}.

As we mentioned in the introduction, a sufficient condition for the convexity of a VGF is that $\Mset\subset\Sbb_+^m$. 
In Section~\ref{sec:polytope}, we discuss more concrete conditions for 
determining convexity when the set~$\Mset$ is a polytope.
In Section~\ref{sec:sufficient}, we describe a more tangible 
sufficient condition for general sets.

\subsection{Convexity with polytope $\Mset$} \label{sec:polytope}

Consider the case where $\Mset$ is a polytope with $p$ vertices, i.e., $\Mset=\op{conv}\{\Mmat_1,\ldots,\Mmat_p\}\,$. 
The support function of this set is given as $S_\Mset(Y) =\max_{i=1,\ldots,p}\,\iprod{Y}{\Mmat_i}\,$ and is piecewise linear
\cite[Section 8.E]{rockafellar2009variational}. 
For a polytope $\Mset$, we define $\Meff$ as a subset of $\{\Mmat_1,\ldots,\Mmat_p\}$ 
with the smallest possible size satisfying 
$S_\Mset(X^TX) = S_{\Meff}(X^TX)$ for all $X\in\R^{n\times m}$.

As an example, for
$\Mset=\{\Mmat:\; \abs{\Mmat_{ij}}\leq \Mbar_{ij}, ~i,j=1,\ldots,m \}$ 
which gives the function defined in \eqref{eq:def-L1Gram}, we have
\begin{align}\label{eq:Meff-L1}
\Meff \subseteq \{\Mmat:\; \Mmat_{ii} = \Mbar_{ii}\;,\;  \Mmat_{ij} = \pm \Mbar_{ij}\;\text{for }i\neq j\, \} \,.
\end{align}
Whether the above inclusion holds with equality or not depends on $n$.

\begin{theorem}	\label{thm:polytope}
For a polytope $\Mset\subset\Sbb^{m}$, the associated VGF is convex if and only if $\Meff\subset\Sbb_+^m$. 
\end{theorem}
\begin{proof}
Obviously, $\Meff\subset\Sbb_+^m$ ensures convexity of $\max_{\Mmat\in\Meff} \tr(X\Mmat X^T) = \Omega_\Mset(X)$. 
Next, we prove necessity of this condition for any $\Meff$. 
Take any $\Mmat_i\in\Meff$. If for every
$X\in\R^{n\times m}$ with $\Omega(X)=\tr(X\Mmat_i X^T)$ there exists another
$\Mmat_j\in\Meff$ with $\Omega(X)=\tr(X\Mmat_j X^T)$, then
$\Meff \backslash \{\Mmat_i\}$ is an effective subset of
$\Mset$ which contradicts the minimality of $\Meff$. Hence, there
exists $X_i$ such that $\Omega(X_i) = \tr(X_i\Mmat_i X_i^T) > \tr(X_i\Mmat_j
X_i^T)$ for all $j\neq i$. 
Hence, 
$\Omega$ is twice continuously differentiable in a small
neighborhood of $X_i$ with Hessian 
$\nabla^2\Omega(\op{vec}(X_i)) = \Mmat_i \otimes \eye_n$,
where $\otimes$ denotes the matrix Kronecker product.
Since $\Omega$ is assumed to be convex, the Hessian has to
be PSD which gives $\Mmat_i\succeq \zero$.  
\end{proof}
 
Next we give a few examples to illustrate the use of Theorem~\ref{thm:polytope}.

\begin{romannum}
\item 
We begin with the example defined in~\eqref{eq:def-L1Gram}. 
Authors in \cite{ZhouXiaoWu11} provided the necessary (when $n\geq m-1$) and
sufficient condition for convexity using results from M-matrix theory:
First, define the comparison matrix $\Mcomp$ associated to the nonnegative 
matrix $\Mbar$ as $\Mcomp_{ii} = {\Mbar_{ii}}$ and
$\Mcomp_{ij} = - {\Mbar_{ij}}$ for $i\neq j\,$. 
Then $\Omega_\Mset$ is convex if $\Mcomp$ is positive semidefinite, 
and this condition is also necessary when $n\geq m-1$ \cite{ZhouXiaoWu11}.
Theorem~\ref{thm:polytope} provides an alternative and more general proof. 
Denote the minimum eigenvalue of a symmetric
matrix~$M$ by $\lambda_\mathrm{min}(M)$. 
From \eqref{eq:Meff-L1}  we have
\begin{align}	
\min_{\Mmat\in\Meff}  \lambda_{\min}(\Mmat) 
&= \min_{\substack{\Mmat\in\Meff\\ \norm{\z}_2=1}} \z^T\Mmat \z 
~\geq~ \min_{\norm{\z}_2=1} \sum_{i}\Mbar_{ii} z_i^2 - \sum_{i\neq j}\Mbar_{ij} \abs{z_iz_j} \nonumber \\
&= \min_{\norm{\z}_2=1} \abs{\z}^T\Mcomp \abs{\z} 
~\geq~ \lambda_{\min}(\Mcomp). 
\label{eq:eigen-Meff-L1}
\end{align}
When $n\geq m-1$, one can construct $X\in\R^{n\times m}$ such that all off-diagonal entries of $X^TX$ are negative (see the example in Appendix A.2 of \cite{ZhouXiaoWu11}). On the other hand, Lemma 2.1(2) of \cite{cai2013note} states that the existence of such a matrix implies $n\geq m-1$. Hence, $\Mcomp \in \Meff$ if and only if $n\geq m-1$. 
Therefore, both inequalities in~\eqref{eq:eigen-Meff-L1} should hold with equality, which means that $\Meff\subset\Sbb_+^m$ if and only if $\Mcomp\succeq 0$.
By Theorem \ref{thm:polytope}, this is equivalent to the VGF in~\eqref{eq:def-L1Gram} being convex. 
If $n<m-1$, then $\Meff$ may not contain $\Mcomp$, thus $\Mcomp\succeq 0$ is only a ``sufficient'' condition for convexity for general $n$.

\begin{figure}[htbp]
\begin{center}
\begin{tikzpicture}[scale=.7]
\begin{axis}[scale only axis, view={143}{10}, hide axis]
\pgfplotsset{
    mycone/.style = {
		colormap={violet}{[1cm] rgb255(0cm)=(25,25,122) color(1cm)=(white) rgb255(5cm)=(238,140,238)},
		fill = red,
		surf,
		shader=interp,opacity=1,z buffer=sort,	mesh/rows=101, 
		}
}

\newcommand{\drawArrows}{
		\addplot3 [->,color=black,solid,dashed] coordinates{(-0.0015,-0.0015,-0.0015) (1.4985,-0.0015,-0.0015)};
		\addplot3 [->,color=black,solid,dashed] coordinates{(-0.0015,-0.0015,-0.0015) (-0.0015,1.2,-0.0015)};
		\addplot3 [->,color=black,solid,dashed] coordinates{(-0.0015,-0.0015,-0.0015) (-0.0015,-0.0015,1.2985)};
		
		\addplot3 [->,color=black,solid,thick] coordinates{(.5,-0.0015,-0.0015) (1.4985,-0.0015,-0.0015)};
		\addplot3 [->,color=black,solid,thick] coordinates{(-0.0015,.4,-0.0015) (-0.0015,1.2,-0.0015)};
		\addplot3 [->,color=black,solid,thick] coordinates{(-0.0015,-0.0015,.5) (-0.0015,-0.0015,1.2985)};
}

\newcommand{\drawBox}{
		\addplot3[opacity=1,area legend,solid,
			colormap/violet, 
			line width=20pt,table/row sep=crcr,
			patch,patch type=rectangle,
			shader=interp,patch table with point meta={
		0	1	2	3	12\\ 
		4	5	6	7	2\\ 
		2	3	7	6	2\\ 
		0	1	5	4	1\\ 
		1	2	6	5	2\\ 
		0	3	7	4	1\\ 
		}]
		table[row sep=crcr] {%
		x	y	z\\
		-0.5	-0.4	-0.5\\
		-0.5	0.4	-0.5\\
		0.5	0.4	-0.5\\
		0.5	-0.4	-0.5\\
		-0.5	-0.4	0.5\\
		-0.5	0.4	0.5\\
		0.5	0.4	0.5\\
		0.5	-0.4	0.5\\
		};
		
		\addplot3[area legend,dashed,line width=.2pt,draw=black,forget plot]
		table[row sep=crcr] {x	y	z\\-0.5	-0.4	-0.5\\0.5	-0.4	-0.5\\};
		\addplot3[area legend,dashed,line width=.2pt,draw=black,forget plot]
		table[row sep=crcr] {x	y	z\\-0.5	-0.4	-0.5\\-0.5	0.4	-0.5\\};
		\addplot3[area legend,dashed,line width=.2pt,draw=black,forget plot]
		table[row sep=crcr] {x	y	z\\-0.5	-0.4	-0.5\\-0.5	-0.4	0.5\\};
		
		\addplot3[area legend,solid,line width=1pt,draw=black,forget plot]
		table[row sep=crcr] {x	y	z\\
		-0.5	0.4	0.5\\
		-0.5	-0.4	0.5\\
		0.5	-0.4	0.5\\
		0.5	-0.4	-0.5\\
		0.5	0.4	-0.5\\
		-0.5	0.4	-0.5\\
		-0.5	0.4	0.5\\
		-0.5	0.4	0.5\\
		0.5	0.4	0.5\\
		0.5	0.4	-0.5\\
		};
		\addplot3[area legend,solid,line width=3.0pt,draw=black,forget plot]
		table[row sep=crcr] {
		x	y	z\\
		0.5	-0.4	0.5\\
		0.5	0.4	0.5\\
		};
		\addplot3[black, mark=*, only marks] coordinates {(0.5,0.4,0.5) (0.5,-0.4,0.5)};
}

\addplot3[mycone] table {PSDcone.txt};  
\drawBox
\drawArrows
\addplot3[mycone] table {PSDconeBack.txt};
\end{axis}
\end{tikzpicture}
\caption{The positive semidefinite cone, and the set in \eqref{eq:Mbox} defined by $\Mbar = [1,~0.8;~0.8,~1]\,$, where 
$2\times 2$ symmetric matrices are embedded into $\mathbb{R}^3\,$. 
The thick edge of the cube is the set of all points 
with the same diagonal elements as $\Mbar$ (see \eqref{eq:Meff-L1}), 
and the two endpoints constitute $\Meff\,$. Positive semidefiniteness of $\Mcomp$ is a necessary and sufficient condition for the convexity of $\Omega_\Mset:\R^{n\times 2}\to \R$ for all $n\geq m-1=1\,$.  
}
\label{fig:boxVGF}
\end{center}
\end{figure}

\item 
Similar to the set $\Mset$ above, consider a box that is not necessarily symmetric around the origin. 
More specifically, let 
$\Mset=\{\Mmat\in \Sbb^m:\; \Mmat_{ii}=D_{ii}\,,\; \abs{M-C} \leq D \}$ 
where $C$ (denoting the center) is a symmetric matrix with zero diagonal, 
and $D$ is a symmetric nonnegative matrix. 
In this case, we have 
$\Meff \subseteq \{\Mmat:\;
\Mmat_{ii}=D_{ii} \,,\; \Mmat_{ij} = C_{ij}\pm D_{ij}~
\mbox{for}~i\neq j\}$.
When used as a penalty function in applications, 
this can capture the prior information that when $\x_i^T\x_j$ is not zero, 
a particular range of acute or obtuse angles
(depending on the sign of $C_{ij}$) between the vectors is preferred. 
Similar to \eqref{eq:eigen-Meff-L1}, 
\begin{eqnarray*}
\min_{\Mmat\in\Meff} \; \lambda_{\min}(\Mmat) 
\geq \min_{\norm{\z}_2=1} \abs{\z}^T\comp{D} \abs{\z} + \z^T C \z
\geq \lambda_{\min}(\comp{D}) + \lambda_{\min}(C) ,
\end{eqnarray*}
where $\comp{D}$ is the comparison matrix associated to $D\,$. Note that $C$
has zero diagonals and cannot be PSD. Hence, a sufficient condition for
convexity of $\Omega_\Mset$ defined by an asymmetric box is that
$\lambda_{\min}(\comp{D}) + \lambda_{\min}(C) \geq 0$. 

\item 
Consider the VGF defined in \eqref{eq:weighted-Linfty}, whose associated
variational set is 
\begin{equation}\label{eq:Mset-linf}
\Mset = \{\Mmat\in\Sbb^m:\; \textstyle\sum_{(i,j):\, \Mbar_{ij}\neq0} \abs{{\Mmat_{ij}}/{\Mbar_{ij}}}\leq 1 \,,\; \Mmat_{ij}=0 \text{ if } \Mbar_{ij}=0 \} ,
\end{equation}
where $\Mbar$ is a symmetric nonnegative matrix. Vertices of $\Mset$ are matrices 
with either only one nonzero value $\Mbar_{ii}$ on the diagonal, or two nonzero off-diagonal entries at
$(i,j)$ and $(j,i)$ equal to $\tfrac{1}{2}\Mbar_{ij}$ or $-\tfrac{1}{2}\Mbar_{ij}\,$. 
The second type of matrices cannot be PSD as their diagonal is zero,  
and according to Theorem~\ref{thm:polytope}, convexity of $\Omega_\Mset$ requires these vertices do not belong to $\Meff\,$.  Therefore, the matrices in $\Meff$ should be diagonal. Hence, a convex VGF corresponding to the set~\eqref{eq:Mset-linf} has the form $\Omega(X) = \max_{i=1,\ldots,m} \Mbar_{ii}\norm{\x_i}_2^2\,$. 
To ensure such a description for $\Meff$ we need $\max\{\Mbar_{ii} \norm{\x_i}_2^2,\Mbar_{jj} \norm{\x_j}_2^2\} \geq \Mbar_{ij}\abs{\x_i^T\x_j}$ for
all $i$, $j$ and any $X\in\R^{n\times m}\,$, which is equivalent to $\Mbar_{ii}\Mbar_{jj}\geq \Mbar_{ij}^2$ for all $i,j\,$.  
This is satisfied if $\Mbar\succeq \zero\,$. 
However, positive semidefiniteness is not necessary. For example, all of three principal minors of the following matrix are nonnegative but it is not PSD: $\Mbar = [1,1,2;1,2,0;2,0,5]\not\succeq 0$.
\end{romannum}
\subsection{A spectral sufficient condition} \label{sec:sufficient}
As mentioned before, it is generally not clear how to provide easy-to-check necessary and sufficient convexity guarantees for the case of non-polytope sets $\Mset$. 
However, simple sufficient conditions 
can be easily checked for certain classes of sets $\Mset$, for example spectral sets (Lemma \ref{lem:lcal-proj}). 
We first provide an example and consider a specialized approach to establish convexity, which 
illustrates the advantage of a simple guarantee as the one we present in Lemma \ref{lem:lcal-proj}. 
\begin{romannum}
\item 
Consider the VGF defined in~\eqref{eq:weighted-Fro} and its associated set given in~\eqref{eqn:Mset-norm} when we plug in the Frobenius norm; i.e.,  
\[
  \Mset = \{ K\circ \Mbar :\; \|K\|_F\leq 1, ~K^T=K\} .
\]
In this case, $\Mset$ is not a polytope, but we can proceed with a similar 
analysis as in the previous subsection.
In particular, given any $X\in\R^{n\times m}$, the value of $\Omega_\Mset(X)$
is achieved by an optimal matrix 
$K_X = (\Mbar\circ X^TX) / \norm{\Mbar\circ X^TX}_F\,$. 
We observe that, 
\[
\Mbar\succeq 0
\implies 
\Mbar\circ \Mbar \succeq 0
\iff 
K_X\circ \Mbar \succeq 0 \,,~\forall X
\implies 
\Omega_\Mset~\text{is convex.}
\]
The first implication is by Schur Product Theorem \cite[Theorem 7.5.1]{horn2012matrix} and does not hold in reverse; e.g., $\Mbar \circ \Mbar = [1,1,2;1,2,3;2,3,5.01] \succeq 0$ while $\Mbar \not\succeq 0$. The second implication, from left to right, is again by Schur Product Theorem. The right to left part is by observing that for any $n\geq 1$, $X$ can always be chosen to select a principal minor of $\Mbar \circ \Mbar$. The third implication is straightforward; pointwise maximum of convex quadratics is convex. All in all, a sufficient condition for $\Omega_\Mset$ being convex is that the Hadamard square of $\Mbar$, namely $\Mbar \circ \Mbar$, is PSD. It is worth mentioning that when $\Mbar \circ \Mbar\succeq 0$, hence real, nonnegative and PSD, it is referred to as a {\em doubly nonnegative matrix}.

\end{romannum}

Denote by $\Mmat_+$ the orthogonal projection of a symmetric matrix $\Mmat$ 
onto the PSD cone, which is given by the matrix formed by only positive 
eigenvalues and their associated eigenvectors of $\Mmat\,$. 
\begin{lemma}[a sufficient condition]\label{lem:lcal-proj}
$\Omega_\Mset$ is convex provided that for any $\Mmat\in\Mset$ there exists $\Mmat' \in\Mset$ such that $\Mmat_+ \preceq \Mmat'\,$. 
\end{lemma}

\begin{proof}
For any $X$, 
$\tr(X \Mmat X^T) \leq \tr(X \Mmat_+ X^T)$ clearly holds. Therefore,
\[
\Omega_\Mset(X)  =\max_{\Mmat\in\Mset} \;\tr(X \Mmat X^T) 
\leq \max_{\Mmat\in\Mset} \;\tr(X \Mmat_+ X^T)\,.
\]
On the other hand, the assumption of the lemma gives 
\[
\max_{\Mmat\in\Mset} \;\tr(X \Mmat_+ X^T) 
\leq \max_{\Mmat'\in\Mset} \;\tr(X \Mmat' X^T)  = \Omega_\Mset(X) 
\]
which implies that the inequalities have to hold with equality, which implies that $\Omega_\Mset(X)$ is convex. Note that the assumption of the lemma can hold while $\Mset_+\not\subseteq \Mset\,$.
\end{proof}

On the other hand, it is easy to see that the condition in Lemma
\ref{lem:lcal-proj} is not necessary. 
Consider $\Mset=\{\Mmat\in\Sbb^2:\; \abs{\Mmat_{ij}}\leq 1\}\,$. 
Although the associated VGF is convex (because the comparison matrix is PSD), 
there is no matrix $\Mmat'\in \Mset$ satisfying $\Mmat'\succeq \Mmat_+$, where 
$\Mmat = [0,1;1,1] \in\Mset$ and $\Mmat_+ \simeq [0.44,0.72; 0.72 ,1.17]$, 
as for any $\Mmat'\in\Mset$ we have $(\Mmat'-\Mmat_+)_{22}<0\,$.

As discussed before, when $\Mset$ is a polytope, convexity of $\Omega_\Mset \equiv  \Omega_{\Meff}$ is equivalent to $\Meff \subset \Sbb_+^m$. For general sets $\Mset$, we showed that $\Mset_+ \subseteq \Mset$ is a sufficient condition for convexity. Similar to the proof of Lemma \ref{lem:lcal-proj}, we can provide another sufficient condition for convexity of a VGF: that all of the maximal points of $\Mset$ with respect to the partial order defined by $\Sbb_+^m$ (the Loewner order) are PSD. 
These are the points $\Mmat\in\Mset$ for which $(\Mset - \Mmat) \cap \Sbb_+^m = \{0_m\}$. 
In all of these pursuits, we are looking for a subset $\Mset'$ of PSD cone such that $\Omega_\Mset \equiv  \Omega_{\Mset'}$. When such a set exists, $\Omega_\Mset$ is convex and many optimization quantities can be computed for it.

Hereafter, we assume there exists a set $\Mset'\subseteq \Mset\cap \mathbb{S}_+$ for which $\Omega_\Mset \equiv \Omega_{\Mset'}$, which in turn implies $\Omega_\Mset \equiv  \Omega_{\Mset\cap \mathbb{S}_+}$. 
For example, based on Theorem \ref{thm:polytope}, this property holds for all convex VGFs associated to a polytope $\Mset$.

\subsection{Conjugate function}
\label{sec:conjugate}
For any function $\Omega\,$, the conjugate function is defined as 
$\Omega^\st(Y) = \sup_X \;\iprod{X}{Y} - \Omega(X)$ 
and the transformation that maps $\Omega$ to $\Omega^\st$ is called the 
Legendre-Fenchel transform (e.g., \cite[Section~12]{Rockafellar70book}).

\begin{lemma}[conjugate VGF]	\label{thm:dual-conj}
Consider a convex VGF associated to a compact convex set $\Mset$ with $\Omega_\Mset \equiv  \Omega_{\Mset\cap \Sbb_+}$. 
The conjugate function is 
\begin{eqnarray}	\label{eq:conj-pseudo}
\Omega_\Mset^\st(Y) = \tfrac{1}{4} \inf_{ \Mmat}  \bigl\{  \tr(Y \Mmat^\dagger Y^T)  \; : \;  \op{range}(Y^T)\subseteq \op{range}(\Mmat)  \, ,\;  \Mmat\in\Mset\cap\Sbb_+^m \bigr\}, 
\end{eqnarray}
where $\Mmat^\dagger$ is the Moore-Penrose pseudoinverse of $\Mmat$. 
\end{lemma}

Note that $\Omega^\st(Y)$ is $+\infty$ if the optimization problem in \eqref{eq:conj-pseudo} is infeasible; i.e., if 
$\op{range}(Y^T)\not\subseteq \op{range}(\Mmat)$ for all $\Mmat\in\Mset\cap\Sbb_+^m$; equivalently, if 
$Y(\eye-\Mmat\Mmat^\dagger)$ is nonzero for all $\Mmat\in\Mset\cap\Sbb_+^m$, where $\Mmat\Mmat^\dagger$ 
is the orthogonal projection onto the range of $\Mmat$. This can be seen using 
generalized Schur complements; e.g., see Appendix A.5.5~in \cite{boyd2004convex} or \cite{burns1974generalized}.

\begin{proof}
By our assumption, that $\Omega_\Mset \equiv  \Omega_{\Mset\cap \Sbb_+}$, we get $\Omega_\Mset^\st\equiv \Omega_{\Mset\cap\Sbb_+}^\st$. 
Define
\begin{align}\label{eq:SDP}
f_\Mset(Y) ~=~ \tfrac{1}{4} \inf_{ \Mmat,\,C} \;  \biggl\{  \tr(C)  \; : \;  \begin{bmatrix}  \Mmat & Y^T \\ Y & C\end{bmatrix} \succeq \zero \, ,\;  \Mmat\in\Mset \biggr\} \,.
\end{align}
The positive semidefiniteness constraint 
implies $M\succeq \zero$, 
therefore 
$f_\Mset \equiv f_{\Mset\cap\Sbb_+}$.
Its conjugate function is
\begin{equation} \begin{aligned} \label{eq:f_M_conj}
f_\Mset^\st(X) 
&~=~ \adjustlimits\sup_Y \sup_{\Mmat,C}\; \biggl\{ \iprod{X}{Y} - \tfrac{1}{4}\tr(C) \; : \;  \begin{bmatrix}  \Mmat & Y^T \\ Y & C\end{bmatrix} \succeq \zero \, ,\;  \Mmat\in\Mset \biggr\} \\ 
&~=~ \adjustlimits\sup_{\Mmat\in\Mset\cap\Sbb_+} \sup_{Y,C}\; \biggl\{ \iprod{X}{Y} - \tfrac{1}{4}\tr(C) \; : \;  \begin{bmatrix}  \Mmat & Y^T \\ Y & C\end{bmatrix} \succeq \zero  \biggr\} \,. 
\end{aligned} \end{equation}
Consider the dual of the inner optimization problem over $Y $and $C$. 
{Let $W\succeq \zero$ be the dual variable with corresponding blocks, and write the Lagrangian as
\[
L(Y,C,W) = \iprod{X}{Y} - \tfrac{1}{4}\tr(C) +\iprod{W_{11}}{\Mmat}+ 2 \iprod{W_{21}}{Y} + \iprod{W_{22}}{C} \,,
\] 
whose 
maximum value is finite only if $W_{21} = -\tfrac{1}{2}X$ and $W_{22}= \tfrac{1}{4}I$. Therefore, the dual problem is 
\[
\min_{W_{11}}\; \biggl\{ \iprod{W_{11}}{\Mmat} :\;  \begin{bmatrix}W_{11}& -\tfrac{1}{2}X^T \\-\tfrac{1}{2}X& \tfrac{1}{4}I\end{bmatrix}\succeq \zero \biggr\} 
= \min_{W_{11}}\; \biggl\{ \iprod{W_{11}}{\Mmat} :\;  W_{11}\succeq X^TX \biggr\},
\]
which is equal to $\iprod{\Mmat}{X^TX}\,$. 
Plugging in \eqref{eq:f_M_conj}, we conclude $f_\Mset^\st\equiv\Omega_{\Mset\cap\Sbb_+}$. 
}

Next, convexity and lower semi-continuity of $f_\Mset$ imply
$f_\Mset^{\st\st}=f_\Mset\,$ 
(e.g., \cite[Theorem~11.1]{rockafellar2009variational}).
Therefore, $f_\Mset$ is equal to $\Omega_{\Mset\cap\Sbb_+}^\st$ which we showed
to be equal to $\Omega_\Mset^\st\,$. 
Taking the generalized Schur complement of the
semidefinite constraint in \eqref{eq:SDP} gives the desired representation in
\eqref{eq:conj-pseudo}.
\end{proof}
Note that \eqref{eq:SDP} is preferred to a representation with $\Mset\cap\Sbb_+$ substituted for $\Mset$. This is because $\Mset$ can have a much simpler representation than $\Mset\cap\Sbb_+$; e.g., 
as for \eqref{eq:Mbox}.

\subsection{Related norms}\label{sec:related_norms}

Given a convex VGF $\Omega_\Mset$, with $\Omega_\Mset \equiv \Omega_{\Mset\cap\Sbb_+}$, we have
\[
\Omega_\Mset(X) = \sup_{\Mmat\in\Mset\cap\Sbb_+}\; \tr(X\Mmat X^T) =  \sup_{\Mmat\in\Mset\cap\Sbb_+}\; \norm{X\Mmat^{1/2}}_F^2 \,.
\]
This representation shows that $\sqrt{\Omega_\Mset}$ is a semi-norm: absolute homogeneity holds, and it is easy to prove the
triangle inequality for the maximum of semi-norms. 
The next lemma, which can be seen from Corollary 15.3.2 of \cite{Rockafellar70book}, generalizes this assertion.
\begin{lemma}	\label{lem:homog-norm}
Suppose a function $\Omega:\R^{n\times m}\to \R$ is homogeneous of order 2, 
i.e., $\Omega(\theta X) = \theta^2\Omega(X)\,$. 
Then its square root $\norm{X} = \sqrt{\Omega(X)}\,$ is a semi-norm if and only
if $\Omega$ is convex. If $\Omega$ is strictly convex then $\sqrt{\Omega}$ is a norm. 
\end{lemma}

\paragraph{Dual Norm} 
Considering $\nor_\Mset \equiv \sqrt{\Omega_\Mset}\,$, we have $\tfrac{1}{2} \Omega_\Mset \equiv \tfrac{1}{2} \nor_\Mset^2\,$. Taking the conjugate function of both sides yields $2\Omega_\Mset^\st \equiv \tfrac{1}{2} (\nor_\Mset^\st)^2$ where we used the order-2 homogeneity of $\Omega_\Mset\,$. 
Therefore, $\nor_\Mset^\st \equiv 2\sqrt{\Omega_\Mset^\st}\,$. 
Given the representation of $\Omega_\Mset^\st$ in Lemma \ref{thm:dual-conj},
one can derive a similar representation for $\sqrt{\Omega_\Mset^\st}$ as follows.
\begin{theorem}
For a convex VGF $\Omega_\Mset$ associated to a nonempty compact convex set $\Mset$, with $\Omega_\Mset \equiv \Omega_{\Mset\cap\Sbb_+}$,
\begin{align} \label{eq:dual-div}
\norm{Y}_\Mset^\st = 2\sqrt{\Omega_\Mset^\st(Y) }
= \tfrac{1}{2}\inf_{\Mmat, C}\; \biggl\{  \tr(C) + \gamma_\Mset(\Mmat)  :\;  \begin{bmatrix}\Mmat& Y^T\\ Y & C\end{bmatrix} \succeq \zero  \biggr\} ,
\end{align}
where $\gamma_\Mset(\Mmat) = \inf \{\lambda\geq 0:\; \Mmat\in\lambda \Mset\}$ is the gauge function associated to $\Mset\,$. 
\end{theorem}
\begin{proof}
The square root function, over positive numbers, can be represented in a variational form as 
$
\sqrt{y} = \min\,\{ \alpha + \frac{y}{4\alpha} :\; \alpha>0\}\,
$.  
Without loss of generality, suppose $\Mset$ is a compact convex set containing the origin. 
Provided that $\Omega_\Mset^\st(Y) >0\,$, from the variational representation of a conjugate VGF function we have
\begin{align*}
\sqrt{\Omega_\Mset^\st(Y) }
&=\tfrac{1}{4} \inf_{ \Mmat, \alpha\geq 0}  \left\{  \alpha + \tfrac{1}{\alpha}\tr(Y \Mmat^\dagger Y^T)  \; : \;  \op{range}(Y^T)\subseteq \op{range}(\Mmat)  \, ,\;  \Mmat\in\Mset\cap\Sbb_+^m \right\} \\
&=\tfrac{1}{4} \inf_{ \Mmat, \alpha\geq 0}  \left\{  \alpha + \tr(Y \Mmat^\dagger Y^T)  \; : \;  \op{range}(Y^T)\subseteq \op{range}(\Mmat)  \, ,\;  \Mmat\in\alpha(\Mset\cap\Sbb_+^m) \right\} 
\end{align*}
where we used $(\alpha\Mmat)^\dagger = \Mmat^\dagger/\alpha$ and performed a change of variable. The last representation is the same as the one given in the statement of the lemma, as the constraint restricts $\Mmat$ to the PSD cone, for which $\gamma_\Mset(\Mmat) = \gamma_{\Mset\cap\Sbb_+}(\Mmat)$. On the other hand, when $\Omega_\Mset^\st(Y) = 0\,$, the claimed representation returns 0 as well because $0\in \Mset$.
\end{proof}

As an example, $\Mset = \{\Mmat \succeq\zero :\; \tr(\Mmat)\leq1\}$ gives $\gamma_\Mset(\Mmat) = \tr(M)$ which if plugged in \eqref{eq:dual-div} yields the well-known semidefinite representation for nuclear norm.

\subsection{Subdifferentials}	\label{sec:subdiff}

In this section, we characterize the subdifferential of VGFs and their
conjugate functions, as well as that of their corresponding norms. 
Due to the variational definition of a VGF where the objective function is linear in $\Mmat$, and the fact that $\Mset$ is assumed to be compact, it is straightforward to
obtain the subdifferential of $\Omega_\Mset$ (e.g., see \cite[Theorem~4.4.2]{hiriart2013convex}).

\begin{proposition}\label{prop:subdiff-div}
For a convex VGF $\Omega_\Mset\equiv \Omega_{\Mset\cap \Sbb_+}$, the subdifferential at $X$ is given by
\begin{align*}
\partial \, \Omega_\Mset(X) 
= \op{conv}\left\{2 X \Mmat  : \; \tr(X\Mmat X^T) = \Omega(X), \; \Mmat\in\Mset\cap\Sbb_+ \right\} .
\end{align*}
For 
the norm $\|X\|_\Mset \equiv \sqrt{\Omega_\Mset}$, we have
$\partial\norm{X}_\Mset=\frac{1}{2\norm{X}_\Mset}\partial \,\Omega_\Mset(X)\,$
if $\Omega_\Mset(X)\neq 0$.
\end{proposition}
As an example, the subdifferential of $\Omega(X) = \sum_{i,j=1}^m \Mbar_{ij} \abs{\x_i^T\x_j}\,$ from \eqref{eq:def-L1Gram}, is given by
\begin{align} 	\label{eq:subdiff-L1gram}
\partial \, \Omega(X) =  \{ 2 X \Mmat : \;
\Mmat_{ij} = \Mbar_{ij}\,\op{sign}(\x_i^T\x_j) \;\; \text{if} \;\; \iprod{\x_i}{\x_j} \neq 0   \;,\;\qquad  \\
 \Mmat_{ii} = \Mbar_{ii \;,\;} \abs{\Mmat_{ij}}\leq\Mbar_{ij} \;\text{ otherwise} \} \nonumber \,.
\end{align}

\begin{proposition}	\label{prop:subdiff-conj}
For a convex VGF $\Omega_\Mset\equiv \Omega_{\Mset\cap \Sbb_+}$, the subdifferential of its conjugate function is given by 
\begin{equation} \begin{aligned}	\label{eq:subdiff-st}
\partial \, \Omega_\Mset^\st(Y) 
&= \big\{ \tfrac{1}{2}(Y \Mmat^\dagger + W)  :\; 
\Omega(Y \Mmat^\dagger + W) =4\Omega^\st(Y)=\tr(Y\Mmat^\dagger Y^T) \,,\, \\ 
&\qquad\qquad\qquad \op{range}(W^T)\subseteq \ker(\Mmat)\subseteq \ker(Y)  \,,\; \Mmat\in\Mset\cap\Sbb_+ \big\} \,.
\end{aligned} \end{equation}
When $\Omega_\Mset^\st(Y)\neq 0$ we have 
$\partial\norm{Y}_\Mset^\st = \frac{2}{\norm{Y}^\st_\Mset}\partial \, \Omega_\Mset^\st(Y)\,$.
\end{proposition}
\begin{proof}
We use the results on subdifferentiation in parametric minimization 
\cite[Section~10.C]{rockafellar2009variational}. 
First, let's fix some notation. 
Throughout the proof, we denote $\tfrac{1}{2}\Omega$ by $\Omega$, and $2\Omega^\st$ by $\Omega^\st$. 
Denote by $\iota_\Mset(\Mmat)$ the indicator function of the set $\Mset$ which is $1$ when $\Mmat\in\Mset$ and $+\infty$ otherwise.  
We use $\Mset$ instead of $\Mset\cap\Sbb_+$ to simplify the notation. Considering 
\[
f(Y,\Mmat) := \begin{cases} \frac{1}{2}\tr(Y \Mmat^\dagger Y^T) & \text{if } \op{range}(Y^T)\subseteq \op{range}(\Mmat)  \\ +\infty & \text{otherwise}  \end{cases}
\]
we have $\Omega^\st(Y) = \inf_{\Mmat}\; f(Y,\Mmat) + \iota_\Mset(\Mmat)\,$. 
For such a function, we can use results in \cite[Theorem~4.8]{burke2015matrix} to show that
\[
\partial f(Y,\Mmat) = \op{conv}\left\{ (Z,-\tfrac{1}{2}Z^TZ):\; Z = Y\Mmat^\dagger + W\,,\; \op{range}(W^T)\subseteq \ker(\Mmat) \right\} \,.
\]
Since $g(Y,\Mmat) := f(Y,\Mmat)+\iota_\Mset(\Mmat)$ is convex, we can use the second part of Theorem 10.13 in \cite{rockafellar2009variational}: for any choice of $\Mmat_0$ which is optimal in the definition of $\Omega^\st(Y)\,$, 
\[
\partial\Omega^\st(Y) = \{ Z:\; (Z, \zero)\in \partial g(Y,\Mmat_0) \} \,.
\]
Therefore, for any $Z\in\partial\Omega^\st(Y)$ we have 
$
  \tfrac{1}{2}Z^TZ \in \partial \iota_\Mset(\Mmat_0)  
  = \{G:~\iprod{G}{\Mmat'-\Mmat_0} \leq 0 \,,\; \forall\,\Mmat'\in\Mset \}
$ 
(Here $\partial \iota_\Mset(\Mmat_0)$ is the normal cone of $\Mset$ at $\Mmat_0$.)
This implies 
$
\tfrac{1}{2}\tr(Z\Mmat' Z^T) \leq \tfrac{1}{2}\tr(Z\Mmat_0 Z^T) 
$ 
for all $\Mmat'\in\Mset\,$. Taking the supremum of the left hand side over all $\Mmat'\in\Mset\,$, we get
\[
\Omega(Z) = \tfrac{1}{2}\tr(Z\Mmat_0 Z^T) 
= \tfrac{1}{2}\tr( Y\Mmat_0^\dagger Y^T) = \Omega^\st(Y) \,,
\]
where the second equality follows from $\op{range}(W^T)\subseteq \ker(\Mmat_0)$ (which is equivalent to  $\Mmat_0 W^T = \zero$). 
Alternatively, for any matrix $Z$ from the right hand side of \eqref{eq:subdiff-st} (after adjustment to our rescaling of definition of $\Omega$ by $\tfrac{1}{2}$), and any $Y'\in\R^{n\times m}$ we have
\[
\Omega^\st(Y')
\geq \iprod{Y'}{Z} - \Omega(Z)  
= \iprod{Y'}{Z} - \Omega^\st(Y) 
= \iprod{Y'-Y}{Z} + \Omega^\st(Y) 
\]
where we used Fenchel's inequality, as well as the characterization of $Z$.
Therefore, $Z\in\partial\Omega^\st(Y)$. This finishes the proof. 
Note that for an achieving $\Mmat$, $\ker(M)\subseteq\ker(Y)$ (i.e., $\op{range}(Y^T)\subseteq \op{range}(\Mmat)$) has to hold for the conjugate function to be defined.
\end{proof}

Since $\partial\Omega^\st(Y)$ is non-empty, for any choice of $\Mmat_0\,$, 
there exists a $W$ such that
$\tfrac{1}{2}(Y\Mmat_0^\dagger+W)\in\partial\Omega^\st(Y)\,$. However, finding
such $W$ is not trivial. 
The following lemma characterizes the subdifferential as the solution set of a convex optimization problem involving $\Omega$ and affine constraints. 
\begin{lemma}	\label{lem:conj-subdiff-as-minz}
Given $Y$ and an optimal $\Mmat_0$, 
\begin{align*}
\partial\Omega^\st(Y) 
&=\op{Arg\,min}_Z\; \left\{ \Omega(Z) :\; Z=\tfrac{1}{2}(Y \Mmat_0^\dagger + W) \,,\; W \Mmat_0 \Mmat_0^\dagger = \zero  \right\},
\end{align*}
where $W \Mmat_0 \Mmat_0^\dagger = \zero$ is equivalent to $\op{range}(W^T)\subseteq \ker(\Mmat_0)$.
\end{lemma}
This is because for all feasible $Z$ we have $\Omega(Z) \geq \op{tr}(Z \Mmat_0 Z^T) = \Omega^\st(Y)$. Moreover, notice that the optimality of $\Mmat_0$ implies $\ker(\Mmat_0) \subseteq\ker(Y)$.

The characterization of the whole subdifferential is helpful for understanding optimality conditions, but algorithms only need to compute a single subgradient, which is easier than computing the whole subdifferential.

\subsection{Composition of VGF and absolute values}
\label{sec:VGF+abs}

The characterization of the subdifferential allows us to establish conditions
for convexity of $\Psi(X) = \Omega(\abs{X})$ defined in
\eqref{eq:def-Omega-abs-X}. 
Our result is based on the following Lemma.
\begin{lemma}	\label{lem:inf-absolute}
Given a function $f:\R^n\to\R\,$, consider $g(\x) = \min_{\y\geq\abs{\x}}\, f(\y)\,$, and $h(\x) = f(\abs{\x})\,$, where the absolute values and inequalities are all entry-wise. Then,  
\begin{enumerate}[(a)]
\item\label{lem:inf-absolute-a} $h^{\st\st}\leq g \leq h\,$.
\item\label{lem:inf-absolute-b} If $f$ is convex then $g$ is convex and $g=h^{\st\st}$. 
\end{enumerate}
\end{lemma}

\begin{proof}
\eqref{lem:inf-absolute-a} In $h^\st(\y) = \sup_\x\; \{\iprod{\x}{\y} - f(\abs{\x})\}\,$, the optimal $x$ has the same sign pattern as $y\,$; hence $h^\st(\y)= \sup_{\x\geq \zero}\; \{\iprod{\x}{\abs{\y}} - f(\x)\} \,$. 
Next, we have 
\begin{eqnarray*}
h^{\st\st}(\z) 
&=& \sup_\y\; \bigl\{\iprod{\y}{\z} - \sup_{\x\geq \zero}\;\{ \iprod{\x}{\abs{\y}} - f(\x)\} \bigr \}  
~=~ \adjustlimits\sup_{\y\geq \zero}\inf_{\x\geq \zero}\; \bigl\{\iprod{\y}{\abs{\z}} - \iprod{\x}{\y} + f(\x) \bigr\} \\
&\leq& \adjustlimits\inf_{\x\geq \zero}\sup_{\y\geq \zero}\; \bigl\{ \iprod{\y}{\abs{\z}} - \iprod{\x}{\y} + f(\x)  \bigr\}
~=~ \adjustlimits\inf_{\x\geq \zero}\sup_{\y\geq \zero}\; \bigl\{ \iprod{\y}{\abs{\z}-\x} + f(\x)  \bigr\} \\
&=& \inf_{\x\geq \abs{\z}}\; f(\x) ~=~ g(\z).
\end{eqnarray*}
This shows the first inequality in part~\eqref{lem:inf-absolute-a}. 
The second inequality follows directly from the definition of~$g$ and~$h$.

\eqref{lem:inf-absolute-b} Consider $\x_1,\x_2\in\R^n$ and $\theta \in[0,1]\,$. 
Suppose $g(\x_i)=f(\y_i)$ for some $\y_i\geq\abs{\x_i}\,$, for $i=1,2\,$. 
In other words, $\y_i$ is the minimizer in the definition of $g(\x_i)$. 
Then, 
\[
\theta\y_1 + (1-\theta)\y_2 \geq \theta\abs{\x_1} + (1-\theta)\abs{\x_2} \geq \abs{\theta\x_1+(1-\theta)\x_2} \,.
\]
By definition of~$g$ and convexity of~$f$
\[
g(\theta\x_1+(1-\theta)\x_2) \leq f(\theta\y_1 + (1-\theta)\y_2) \leq \theta f(\y_1) +(1-\theta) f(\y_2) = \theta g(\x_1) +(1-\theta) g(\x_2),
\]
which implies that $g$ is convex.
It is a classical result that the epigraph of the biconjugate $h^{**}$ is the 
closed convex hull of the epigraph of~$h$; in other words, $h^{**}$ is the 
largest lower semi-continuous convex function that is no larger than~$h$
(e.g., \cite[Theorem~12.2]{Rockafellar70book}).
Since $g$ is convex and $h^{\st\st}\leq g \leq h\,$, 
we must have $h^{\st\st}=g\,$. 
\end{proof}

\begin{corollary}	\label{cor:Psi-convex}
Let $\Omega_\Mset$ be a convex VGF. 
Then, $\Omega_\Mset(\abs{X})$ is a convex function of $X$ if and only if
$\Omega_\Mset(\abs{X}) = \min_{Y\geq \abs{X}}\; \Omega_\Mset(Y)\,$.  
\end{corollary}
\begin{proof}
Let $\Omega_\Mset$ be the function~$f$ in Lemma~\ref{lem:inf-absolute}. 
Then we have $h(X)=\Omega_\Mset(|X|)$ and $g(X)=\min_{Y\geq |X|} \Omega_\Mset(Y)$. 
Since here $h$ is a closed convex function, we have $h=h^{\st\st}$ 
\cite[Theorem~12.2]{Rockafellar70book}, thus
part~(a) of Lemma \ref{lem:inf-absolute} implies $h=g\,$. 
On the other hand, given a convex function~$f$,  
part~(b) of Lemma \ref{lem:inf-absolute} states that $g=h^{**}$ is also convex.
Hence, $h=g$ implies convexity of $h\,$. 
\end{proof}

Another proof of Corollary \ref{cor:Psi-convex}, in the case where $\sqrt{\Omega_\Mset}$ is a norm and not a semi-norm, is given as Lemma \ref{lem:Psi-convex-via-norms} in the Appendix.

\begin{lemma}
Let $\Omega_\Mset\,$ be a convex VGF with $\Omega_\Mset \equiv \Omega_{\Mset\cap\Sbb_+}$. 
If $\partial \Omega_\Mset(X) \cap \R^{n\times m}_+ \neq \emptyset$ 
holds for any $X\geq \zero\,$, then $\Psi(X)=\Omega_\Mset(\abs{X})$ is convex. 
\end{lemma}
\begin{proof}
Using the definition of subgradients for $\Omega$ at $\abs{X}$ we have
\[
\Omega(\abs{X}+\Delta) \geq \Omega(\abs{X}) + \sup \{\iprod{G}{\abs{X}+\Delta} :\; G\in \partial \Omega\text{ at } \abs{X} \},
\]
where the right-most term is the directional derivative of $\Omega$ at $\abs{X}$ in the direction $\Delta\,$. 
From the assumption, we get $\Omega(Y)\geq \Omega(\abs{X})$ for all $Y\geq \abs{X}\,$. Therefore, $\Psi(X)=\Omega_\Mset(\abs{X}) =  \min_{Y\geq \abs{X}}\; \Omega_\Mset(Y)\,$. Corollary \ref{cor:Psi-convex} establishes the convexity of $\Psi\,$. 
\end{proof}

For example, consider the VGF $\Omega_\Mset$ defined in~\eqref{eq:def-L1Gram}, 
and assume that it is convex.
Its subdifferential $\partial \Omega_\Mset$ given in \eqref{eq:subdiff-L1gram}. 
For each $X\geq 0$, the matrix product $X\Mbar\geq\zero$ since $\Mbar$ is also
a nonnegative matrix, hence it belongs to $\partial\Omega_\Mset(X)$. 
Therefore the condition in the above lemma is satisfied, 
and the function $\Psi(X)=\Omega_\Mset(|X|)$ is
convex and has an alternative representation 
$\Psi(X)=\min_{Y\geq \abs{X}}\; \Omega_\Mset(Y)\,$.  
This specific function $\Psi$ has been used in \cite{vervier2014learning}
for learning matrices with disjoint supports.

\section{Proximal operators}	\label{sec:prox}
The proximal operator of a closed convex function $h(\cdot)$ is defined as 
$
\prox{h}{\x} = \argmin_\bd{u} \; 
\{ h(\bd{u}) + \tfrac{1}{2}\norm{\bd{u}-\x}_2^2 \} 
$,  
which always exists and is unique (e.g., \cite[Section~31]{Rockafellar70book}).
Computing the proximal operator is the essential step in the proximal point algorithm (\cite{Martinet70ppa,Rockafellar76ppa}) 
and the proximal gradient methods (e.g., \cite{Nesterov13Composite}).
In each iteration of such algorithms, we need to compute $\prox{\tau h}$ 
where~$\tau>0$ is a step size parameter. 
To simplify the presentation, assume $\Mset\subset\Sbb_+^m$ and consider the associated VGF. Then, 
\begin{align}	\label{eq:omega-prox}
\prox{\,\tau\Omega}{X} 
= \adjustlimits\argmin_{Y}  \max_{\Mmat\in\Mset}~ \tfrac{1}{2}\norm{Y-X}_F^2 + \tau \tr(Y \Mmat Y^T)  .
\end{align}
Since $\Mset\subset \Sbb_+$ is a compact convex set, one can change the order of $\min$ and $\max$ 
and first solve for $Y$ in terms of any given~$X$ and~$M$, which gives
$Y=X(I+2\tau M)^{-1}$. 
Then we can find the optimal $M_0\in\Mset$ given $X$ as
\[
\Mmat_0 = \argmin_{\Mmat\in\Mset}\;    \tr\left(X  (\eye+ 2\tau \Mmat)^{-1}  X^T \right)
\]
which gives 
$\prox{\,\tau\Omega}{X} = X(\eye+2\tau\Mmat_0)^{-1}\,$. 
To compute the proximal operator for the conjugate function $\Omega^*\,$, one can use Moreau's formula (see, e.g., \cite[Theorem~31.5]{Rockafellar70book}):
\begin{equation}\label{eq:Moreau-theorem}
  \prox{\,\tau\Omega}{X} +\tau^{-1}\prox{\,\tau^{-1}\Omega^\st}{X} = X \,.
\end{equation}

Next we discuss proximal operators of norms induced by VGFs (section \ref{sec:related_norms}). 
Since computing the proximal operator of a norm is related to projection
onto the dual norm ball, 
i.e., $\prox{\,\tau\nor}{X} =X- \proj_{\nor^\st\leq\tau}(X)$, we can express the proximal operator of the norm
$\nor\equiv \sqrt{\Omega_\Mset(\cdot)}$ as
\begin{align*}
\prox{\,\tau\nor}{X} 
=X- \adjustlimits\argmin_Y \min_{\Mmat, C} 	\biggl\{ \norm{Y-X}_F^2 :\,  \tr(C) \leq \tau^2 ,\, \begin{bmatrix}  \Mmat & Y^T \\ Y & C\end{bmatrix} \succeq 0 ,\, \Mmat\in\Mset   \biggr\},
\end{align*}
using 
\eqref{eq:SDP}, \eqref{eq:dual-div}. Moreover, plugging \eqref{eq:dual-div} in the definition of proximal operator gives
\begin{align*}
\prox{\,\tau\nor_\Mset^\st}{X} 
&= \adjustlimits\argmin_Y\min_{\Mmat,C}\biggl\{  \norm{Y-X}_F^2 + \tau (\tr(C)+\gamma_\Mset(\Mmat)) :\; \begin{bmatrix} \Mmat& Y^T\\ Y & C \end{bmatrix} \succeq\zero  \biggr\},
\end{align*}
where $\gamma_\Mset(\Mmat) = \inf \{\lambda\geq 0:\; \Mmat\in\lambda \Mset\}$ is the gauge function associated to the nonempty convex set $\Mset\,$. 
The computational cost for computing proximal
operators can be high in general (involving solving semidefinite programs); 
however, 
they may be simplified for special cases of $\Mset\,$. 
For example, a fast algorithm for computing the proximal operator of
the VGF associated with the set $\Mset$ defined 
in~\eqref{eq:M-sliced-box-spectral} is presented in \cite{mcdonald2014new}.
For general problems, due to the convex-concave saddle point structure 
in~\eqref{eq:omega-prox}, we may use the mirror-prox algorithm
\cite{nemirovski2004prox} to obtain an inexact solution.

\paragraph{Left unitarily invariance and QR factorization} 
As mentioned before, VGFs and their conjugates are left unitarily invariant. We can use this fact to simplify the computation of corresponding proximal operators when $n\geq m\,$. Consider the QR decomposition of a matrix $Y=QR\,$ where $Q$ is an orthogonal matrix with $Q^TQ=QQ^T=\eye$ and $R=[R_Y^T ~ \zero]^T$ is an upper triangular matrix with $R_Y\in\R^{m\times m}$. From the definition, we have $\Omega(Y)= \Omega(R_Y)$ and $\Omega^\st(Y)= \Omega^\st(R_Y)$. 
For the proximal operators, we can simply plug in $R_X$ from the QR decomposition $X=Q[R_X^T~ \zero]^T$ to get
\begin{align*}
\prox{\,\tau\Omega^\st}{X} 
&= \adjustlimits\argmin_{Y} \min_{\Mmat,C}\;   \biggl\{ \norm{Y-X}_2^2 +\tfrac{1}{2}\tau\tr(C) :\; \begin{bmatrix}\Mmat& Y^T\\ Y& C\end{bmatrix} \succeq \zero \,,\; \Mmat\in\Mset \biggr\}\\
&= Q\cdot \adjustlimits\argmin_{R} \min_{\Mmat,C}\;   \biggl\{ \norm{R-R_X}_2^2 +\tfrac{1}{2}\tau\tr(C) :\; \begin{bmatrix}\Mmat& R^T\\ R& C\end{bmatrix} \succeq \zero \,,\; \Mmat\in\Mset \biggr\}
\end{align*}
where $R$ is constrained to the set of upper triangular matrices and the new PSD matrix is of size $2m$ instead of $n+m$ that we had before. The above equality uses two facts. First, 
\begin{align}\label{eq:rotation-R}
\begin{bmatrix} \eye_m &\zero\\\zero& Q^T\end{bmatrix}\begin{bmatrix}  \Mmat & Y^T \\ Y & C\end{bmatrix} \begin{bmatrix} \eye_m &\zero\\\zero& Q\end{bmatrix}= \begin{bmatrix} \Mmat & R^T\\ R& Q^T C Q \end{bmatrix}\succeq \zero 
\end{align}
where the right and left matrices in the multiplication are positive definite. 
Secondly, $\tr(C) = \tr(C')$ where $C' = Q^TCQ$ and assuming $C'$ to be zero outside the first $m\times m$ block can only reduce the objective function. Therefore, we can ignore the last $n-m$ rows and columns of the above PSD matrix.

More generally, because of left unitarily invariance, the optimal $Y$'s in all of the optimization problems in this section have the same column space as the input matrix $X$; otherwise, a rotation as in \eqref{eq:rotation-R} produces a feasible $Y$ with a smaller value for the objective function.

\section{Algorithms for optimization with VGF} \label{sec:opt-algm}

In this section, we discuss optimization algorithms for solving convex minimization problems, in the form of~\eqref{eq:loss+omega}, with VGF penalties.
The proximal operators of VGFs we studied in the previous section are the
key parts of proximal gradient methods 
(see, e.g., \cite{BeckTeboulle09, BeckTeboulle10, Nesterov13Composite}).
More specifically, when the loss function~$\Loss(X)$ is smooth, 
we can iteratively update the variables $X^{(t)}$ as follows:
\[
  X^{(t+1)} = \prox{\gamma_t\Omega}{X^{(t)} - \gamma_t \nabla\Loss(X^{(t)})},
  \qquad t=0, 1, 2, \ldots,
\]
where $\gamma_t$ is a step size at iteration~$t$.
When~$\Loss(X)$ is not smooth, then we can use subgradients of $\Loss(X^{(t)})$
in the above algorithm, or use the classical subgradient method on the 
overall objective $\Loss(X)+\lambda\Omega(X)$.
In either case, we need to use diminishing step size and 
the convergence can be very slow. 
Even when the convergence is relatively fast 
(in terms of number of iterations), the computational cost of the 
proximal operator in each iteration can be very high.

In this section, we focus on loss functions that have a special form shown in \eqref{eq:loss-variational}. 
This form comes up in many common loss functions, some of which listed later in this section, and allows for faster algorithms. 
We assume that the loss function~$\Loss$ in~\eqref{eq:loss+omega} has the
following representation:
\begin{equation}	\label{eq:loss-variational}
\Loss(X) = \max_{\g\in\cal{G}}\;\; \iprod{X}{\cal{D}(\g)} - \Lhat(\g) \,,
\end{equation}
where $\Lhat :\R^p\to\R$ is a convex function, 
$\cal{G}$ is a convex and compact subset of~$\R^p$, 
and $\cal{D}:\R^p\to\R^{n\times m}$ is a linear operator. 
This is also known as a Fenchel-type representation 
(see, e.g.,  \cite{juditsky2013solving}). Moreover, 
consider the infimal post-composition \cite[Definition~12.33]{bauschke2011convex} of 
$\Lhat:\cal{G}\to\R$ by $\cal{D}(\cdot)\,$, defined as 
\begin{align*}
(\cal{D} \triangleright \Lhat)(Y) = \inf\,\{\Lhat(G):\; \cal{D}(G)= Y\,,\; G\in\cal{G}\} \,.
\end{align*}
Then, the conjugate to this function is equal to $\Loss\,$. 
In other words, $\Loss(X) = \Lhat^\st(\cal{D}^\st(X))$ where $\Lhat^\st$ is the conjugate function and $\cal{D}^\st$ is the adjoint operator. 
The composition of a nonlinear convex loss function and a linear operator 
is very common for optimization of linear predictors in 
machine learning (e.g., \cite{HTFbook}),
which we will demonstrate with several examples later in this section. 

With the variational representation of~$\Loss$ in \eqref{eq:loss-variational}, and assuming $\Omega_\Mset \equiv \Omega_{\Mset\cap \Sbb_+}$, we can write the VGF-penalized loss minimization problem \eqref{eq:loss+omega}
as a convex-concave saddle-point optimization problem:
\begin{align} \label{eq:loss+omega-var} 
J_\op{opt} = \min_{X} \max_{\Mmat\in\Mset\cap\Sbb_+, \,\g\in\cal{G}}\;\; 
 \iprod{X}{\cal{D}(\g)} - \Lhat(\g) + \lambda\tr(X\Mmat X^T) \,.
\end{align} 
If $\Lhat$ is smooth (while $\Loss$ may be nonsmooth) and the sets
$\cal{G}$ and $\Mset$ are simple (e.g., admitting simple projections), 
we can solve problem~\eqref{eq:loss+omega-var} using the 
\emph{mirror-prox} algorithm \cite{nemirovski2004prox,juditsky2013solving}.
In section~\ref{sec:mirror-prox}, we present a variant of the
mirror-prox algorithm equipped with an adaptive line search scheme. 
Then in Section \ref{sec:reduced-form}, we present a preprocessing technique
to transform problems of the form \eqref{eq:loss+omega-var} 
into smaller dimensions, which can be solved more efficiently
under favorable conditions.

Before diving into the algorithmic details, 
we examine some common loss functions and derive the corresponding 
representation~\eqref{eq:loss-variational} for them. 
This discussion will provide intuition for the linear operator~$\cal{D}$ and the set~$\cal{G}$ in relation with data and prediction. 

\paragraph{Norm loss} 
Given a norm $\nor$ and its dual $\nor^\st\,$, consider the squared norm loss 
\[
\Loss(\x) = \tfrac{1}{2} \norm{A\x-\bd{b}}^2 
= \max_{\g}\;  \iprod{\g}{A\x-\bd{b}} - \tfrac{1}{2}(\norm{\g}^\st)^2\,.
\]
In terms of the representation in~\eqref{eq:loss-variational}, 
here we have $\cal{D}(\g) = A^T\g\,$ and
$\Lhat(\g) = \tfrac{1}{2}(\norm{\g}^\st)^2 + \bd{b}^T\g$.
Similarly, a norm loss can be represented as
\[
\Loss(\x) = \norm{A\x-\bd{b}} = \max_\g \; \{ \iprod{\x}{A^T\g} - \bd{b}^T\g :\;  \norm{\g}^\st \leq 1\},
\]
where we have $\cal{D}(\g)=A^T \g$, $\Lhat(\g)=\bd{b}^T\g$ and 
$\cal{G}=\{\g:\|\g\|^*\leq 1\}$.

\paragraph{$\eps$-insensitive (deadzone) loss} 
Another variant of the absolute loss function is called the $\eps$-insensitive loss (e.g., see \cite[Section~14.5.1]{murphy2012machine} for more details and applications) and can be represented, similar to \eqref{eq:loss-variational}, as
\begin{align*}
\Loss_\eps(x) = (\abs{x}-\eps)_+ &= \max_{\alpha,\beta}\; \{ \alpha(x-\eps) + \beta(-x-\eps):\; \alpha,\beta \geq 0, \; \alpha+\beta\leq 1  \} .
\end{align*}

\paragraph{Hinge loss for binary classification} 
In binary classification problems, we are given a set of training examples
$(\bd{a}_1, b_1), \ldots, (\bd{a}_N, b_N)$, where each $\bd{a}_s\in\R^p$ 
is a feature vector and $b_s\in\{+1, -1\}$ is a binary label.
We would like to find $\x\in\R^p$ such that the linear function 
$\bd{a}_s^T \x$ can predict the sign of label~$b_s$ for each $s=1,\ldots,N$.
The hinge loss $\max\{0, 1-b_s(\bd{a}_s^T \x)\}$ returns~$0$ if
$b_s(\bd{a}_s^T \x)\geq 1$ and a positive loss growing with the absolute
value of $b_s(\bd{a}_s^T \x)$ when it is negative.
The average hinge loss over the whole data set can be expressed as
\[
\Loss(\x) 
~= ~\frac{1}{N}\textstyle\sum_{s=1}^N \max\left\{0, 1-b_s (\bd{a}_s^T \x) \right\}
~= ~\max_{\g\in\cal{G}}\; \iprod{\g}{\one-\bd{D}\x} \,.
\]
where $\bd{D} = [b_1\bd{a}_1,\,\ldots,\, b_N \bd{a}_N ]^T\,$. 
Here, in terms of~\eqref{eq:loss-variational},
we have, $\cal{G} = \{\g\in\mathbb{R}^N:\; 0\leq g_s\leq 1/N \,\}\,$, $\mathcal{D}(\g) = -\bd{D}^T\g\,$, and $\Lhat(\g) = -\one^T\g\,$. 

\paragraph{Multi-class hinge loss} 
For multiclass classification problems, each sample $\bd{a}_s$ has a label
$b_s\in\{1,\ldots,m\}$, for $s=1,\ldots,N$.
Our goal is to learn a set of classifiers $\x_1,\ldots, \x_m$, 
that can predict the labels $b_s$ correctly. 
For any given example $\bd{a}_s$ with label $b_s$, 
we say the prediction made by $\x_1,\ldots,\x_m$ is correct if
\begin{equation}\label{eq:multiclass-constraints}
  \x_i^T\bd{a}_s \geq \x_j^T\bd{a}_s \quad \text{for all } (i, j) \in\cal{I}(b_s),
\end{equation}
where $\cal{I}_{k}\,$, for $k=1,\ldots,m\,$, characterizes the required 
comparisons to be made for any example with label~$k$.
Here are two examples.
\begin{remunerate}
  \item 
  \emph{Flat multiclass classification:} $\cal{I}(k) = \{(k, j): j\neq k\}$. 
  In this case, the constraints in~\eqref{eq:multiclass-constraints} are 
  equivalent to the label $b_s = \argmax_{i\in\{1,\ldots,m\}} \x_i^T \bd{a}_s$;
  see \cite{WestonWatkins}.
  \item 
  \emph{Hierarchical classification.} In this case, the labels $\{1,\ldots,m\}$
  are organized in a tree structure, and each $\cal{I}(k)$ is a special subset
  of the edges in the tree depending on the class label~$k$; 
  see Section~\ref{sec:hier-class} and 
  \cite{DekKesSin,ZhouXiaoWu11} for further details.
\end{remunerate}

Given the labeled data set $(\bd{a}_1,b_1),\ldots,(\bd{a}_N, b_N)$, we can 
optimize $X=[\x_1,\ldots,\x_m]$ to minimize the averaged multi-class hinge loss
\begin{equation}\label{eq:multiclass-hinge}
  \Loss(X) = 
\frac{1}{N}\textstyle\sum_{s=1}^N\,\max \bigl\{0, 1- \max_{(i,j)\in \cal{I}(b_s)} 
\{\x_i^T\ba_s - \x_j^T\ba_s\}\bigr\},
\end{equation}
which penalizes the amount of violation for the inequality constraints
in~\eqref{eq:multiclass-constraints}.

In order to represent the loss function in~\eqref{eq:multiclass-hinge}
in the form of~\eqref{eq:loss-variational}, we need some more notations.
Let $p_k= \abs{\cal{I}(k)}$, and define $E_k\in\R^{m\times p_k}$ as the 
incidence matrix for the pairs in $\cal{I}_k\,$;  i.e., 
each column of $E_k$, corresponding to a pair $(i,j)\in\cal{I}_k$, has only 
two nonzero entries: $-1$ at the $i$th entry and $+1$ at the $j$th entry.
Then the $p_k$ constraints in \eqref{eq:multiclass-constraints}
can be summarized as $ E_k^T X^T \bd{a}_s \leq \zero$.
It can be shown that the multi-class hinge loss $\Loss(X)$ 
in~\eqref{eq:multiclass-hinge} can be represented in the 
form~\eqref{eq:loss-variational} via
\[
\cal{D}(\g) = -A \, \cal{E}(\g), \qquad \text{and} \qquad 
\Lhat(\g) = -\one^T\g ,
\]
where $A=[\bd{a}_1~\cdots~\bd{a}_N]$ and 
$\cal{E}(\g) = [E_{b_1}\g_1~\cdots~ E_{b_N}\g_N]^T \in\R^{N\times m}$. 
Moreover, the domain of maximization in~\eqref{eq:loss-variational} 
is defined as
\begin{equation}	\label{eq:Gc-hinge}
\cal{G} = \cal{G}_{b_1}\times \ldots \times \cal{G}_{b_N} \quad 
\text{where} \quad 
\cal{G}_k =  \{\g\in\R^{p_k} \,:\; \g \geq 0 \;,\; \one^T \g \leq 1/N \}\,.
\end{equation}
Combining the above variational form for multi-class hinge loss and a VGF as
penalty on~$X$, we can reformulate the nonsmooth convex optimization problem
  $\min_X\; \{\Loss(X) + \lambda \Omega_\Mset(X) \}$
as the convex-concave saddle point problem
\begin{equation}	\label{eq:hinge-divnorm-saddle}
\adjustlimits\min_{X} \max_{\Mmat\in\Mset\cap\Sbb_+,\, \g\in\cal{G}}\; 
\one^T\g - \iprod{X}{A\, \cal{E}(\g)}  + \lambda\tr(X\Mmat X^T) .
\end{equation}

\subsection{Mirror-prox algorithm with adaptive line search}
\label{sec:mirror-prox}

The mirror-prox (MP) algorithm was proposed by Nemirovski
\cite{nemirovski2004prox} for approximating the saddle points of smooth
convex-concave functions and solutions of variational inequalities with
Lipschitz continuous monotone operators. 
It is an extension of the extra-gradient method \cite{Korpelevic76},
and more variants are studied in \cite{JuditskyNemirovski11chapter6}. 
In this section, we first present a variant of the MP algorithm equipped with
an adaptive line search scheme. 
Then explain how to apply it to solve the VGF-penalized loss minimization 
problem~\eqref{eq:loss+omega-var}.

\newcommand{\Zc}{\cal{Z}}
\newcommand{\Es}{\mathscr{E}}
We describe the MP algorithm in the more general setup of solving variational
inequality problems.
Let~$\Zc$ be a convex compact set in Euclidean space~$\Es$ equipped with inner
product $\langle\cdot,\cdot\rangle$, and $\|\cdot\|$ and $\|\cdot\|_*$ be a pair of dual norms on~$\Es$,
i.e., $\|\xi\|_*=\max_{\|z\|\leq 1}\langle\xi,z\rangle$. 
Let $F:\Zc\to \Es$ be a Lipschitz continuous monotone mapping:
\begin{align}
	    \forall\, z, z'\in \Zc: \;\; \|F(z)-F(z')\|_*\leq L\|z-z'\|, \;\; \text{and, }\;  \langle F(z)-F(z'),z-z'\rangle\geq 0\,.
	    \label{eqn:Lipschitz-monotone}
\end{align}
The goal of the MP algorithm is to approximate a
(strong) solution to the variational inequality associated with $(\Zc,F)$: 
$\langle F(z^\star), z - z^\star \rangle \geq 0, \; \forall\, z\in \Zc$. 
Let $\phi(x,y)$ be a smooth function that is convex in~$x$ and concave in~$y$,
and~$\cal{X}$ and~$\cal{Y}$ be closed convex sets. Then the convex-concave saddle point problem
\begin{equation*} 
    \min_{x\in \cal{X}} \max_{y\in \cal{Y}} ~\phi(x,y), 
\end{equation*}
can be posed as a variational inequality problem with
$z=(x,y)$, $\Zc=\cal{X}\times \cal{Y}$ and 
\begin{equation}\label{eq:grad-mapping}
    F(z) = \left[ \begin{array}{rr}
            \nabla_x\phi(x,y) \\ -\nabla_y\phi(x,y) \end{array}
    \right].
  \end{equation}

The setup of the mirror-prox algorithm requires a distance-generating function
$h(z)$ which is compatible with the norm $\|\cdot\|$. In other words, 
$h(z)$ is subdifferentiable on the relative interior of~$\Zc$, denoted $\Zc^o$,
and is strongly convex with modulus~1 with respect to~$\|\cdot\|$, i.e., 
for all $z, z'\in \Zc$, we have $\langle \nabla h(z) -\nabla h(z'), z-z' \rangle \geq \|z-z'\|^2$. 
For any $z\in \Zc^o$ and $z'\in \Zc$, we can define the Bregman divergence at~$z$ as
\[
    V_z(z') = h(z') - h(z) - \langle\nabla h(z), z'-z\rangle ,
\]
and the associated proximity mapping as
\begin{align*}
  \proxb{z}{\xi}
~=~ \argmin_{z'\in \Zc} \left\{ \langle \xi, z'\rangle + V_z(z') \right\} 
~=~ \argmin_{z'\in \Zc} \left\{ \langle \xi-\nabla h(z), z'\rangle+h(z')\right\}.
\end{align*}
With these definitions, we are now ready to present the MP algorithm
in Figure~\ref{alg:mirror-prox}.
Compared with the original MP algorithm 
\cite{nemirovski2004prox,JuditskyNemirovski11chapter6}, our variant 
employs an adaptive line search procedure to 
determine the step sizes $\gamma_t\,$, for $t=1,2,\ldots\,$.
We can 
exit the algorithm whenever $V_{z_t}(z_{t+1}) \leq \epsilon$
for some $\epsilon>0$.
Under the assumptions in~\eqref{eqn:Lipschitz-monotone}, 
the MP algorithm in Figure~\ref{alg:mirror-prox} enjoys the same $O(1/t)$ 
convergence rate as the one proposed in \cite{nemirovski2004prox}, but 
performs much faster in practice.
The proof requires only simple modifications of the proof in 
\cite{nemirovski2004prox,JuditskyNemirovski11chapter6}. 
\begin{figure}[ht]
\begin{center}
\fbox{
\begin{minipage}[c][14em][c]{0.5\textwidth}{
{\bf Algorithm:} Mirror-Prox$(z_1,\gamma_1,\eps)$ \vskip-.8em
\begin{algorithmic}
\REPEAT
	\STATE $t:=t+1$
	\REPEAT
  \STATE $\gamma_t := \gamma_t/c_\mathrm{dec}$
		\STATE $w_t := \proxb{z_t}{\gamma_t F(z_t)}$
		\STATE $z_{t+1} := \proxb{z_t}{\gamma_t F(w_t)}$
	\UNTIL {$\delta_t\leq 0$ }
  \STATE $\gamma_{t+1}:= c_\mathrm{inc} \gamma_t$
\UNTIL {$V_{z_t}(z_{t+1}) \leq \eps$}
\RETURN $\bar z_t := (\sum_{\tau=1}^t \gamma_\tau)^{-1} \sum_{\tau=1}^t \gamma_\tau w_\tau$
\end{algorithmic}
}\end{minipage}}
\end{center}
\caption{Mirror-Prox algorithm with adaptive line search. 
  Here $c_\mathrm{dec}>1$ and $c_\mathrm{inc}>1$ are parameters 
  controlling the decrease and increase of the step size~$\gamma_t$ 
  in the line search trials. 
  The stopping criterion for the line search is $\delta_t\leq 0$ where
   $\delta_t = \gamma_t \langle F(w_t), w_t-z_{t+1} \rangle
   - V_{z_t}(z_{t+1}) \,$.}
\label{alg:mirror-prox}
\end{figure}

When $\Lhat$ is smooth and $\Omega_\Mset\equiv \Omega_{\Mset\cap\Sbb_+}$, we can apply MP algorithm to solve the saddle-point problem in~\eqref{eq:loss+omega-var}. 
Then, the gradient mapping in~\eqref{eq:grad-mapping} becomes
\begin{align}\label{eq:F-MP-vgf}
F(X,\Mmat,\g) =  \begin{bmatrix} \op{vec}(2\lambda X \Mmat +\cal{D}(\g)) \\ -\lambda\op{vec}(X^TX) \\ \op{vec}(\nabla \Lhat(\g) - \cal{D}^\st(X))  \end{bmatrix} , 	
\end{align}
where $\cal{D}^\st(\cdot)$ is the adjoint operator to $\cal{D}(\cdot)$. 
Assuming $\g\in\R^p\,$, computing $F$ requires $O(nm^2+nmp)$ operations for
matrix multiplications. In Section \ref{sec:reduced-form}, we present a method that can potentially reduce the problem size by replacing $n$ with $\min\{mp,n\}\,$. 
In the case of SVM with the hinge loss 
as in our real-data numerical example, one can replace $n$ by
$\min\{N,mp,n\}\,$, where $N$ is the number of samples.

The assumption $\Omega_\Mset\equiv \Omega_{\Mset\cap\Sbb_+}$ provides us with a convex-concave saddle point optimization problem in \eqref{eq:loss+omega-var}. 
However, mirror-prox iterations for \eqref{eq:loss+omega-var} require a projection onto $\Mset\cap\Sbb_+$ (or more generally, computation of the proximity mapping $\proxb{z}{\xi}$ corresponding to the mirror map we use and a set $\Zc$ defined via $\Mset\cap\Sbb_+$), and such projections might be much more complicated than projection onto $\Mset$. 
In fact, while $\Omega_\Mset\equiv \Omega_{\Mset\cap\Sbb_+}$ implies that the achieving matrix in $\sup_{\Mmat\in\Mset} \langle\Mmat, X^TX \rangle$ is always in $\Mset\cap\Sbb_+$, we need a separate guarantee to be able to project onto $\Mset$ and $\Mset\cap\Sbb_+$ interchangeably. 
We remark on a guarantee for this in the following, where Lemma \ref{lem:proj-Z} and Corollary \ref{cor:proj-PSD} provide sufficient conditions for when projection of a PSD matrix onto $\Mset$ is equivalent to projection onto $\Mset \cap \Sbb_+$. 
\begin{lemma}\label{lem:proj-Z}
For any $G\succeq 0$, consider $P = \proj_\Mset (G)$ and its Moreau decomposition with respect to the positive semidefinite cone as $P = P_{+} - P_{-}$ where $P_{+} , P_{-} \succeq 0$ and $\langle P_{+} , P_{-}\rangle =0 $. Then, $P_{+}\in\Mset$ implies $P_{-} =0\,$. 
\end{lemma}
\begin{proof}
	Recall the firm nonexpansive property of the projection operator onto a convex set \cite{rockafellar2009variational} applied to $P = \proj_\Mset (G)$ and $P_{+} = \proj_\Mset (P_{+})$ (implied by $P_+\in\Mset$). We get $\norm{P - P_{+}}_F^2 \leq \langle P - P_{+} , G - P_{+}\rangle$ 
		which implies $\langle P_{-} , G \rangle + \norm{P_{-}}_F^2 \leq 0$. 
		Moreover, for two PSD matrices $G$ and $P_{-}$ we have $\langle G, P_{-} \rangle \geq 0$. All in all, $P_{-} =0$. 
\end{proof}
\begin{corollary}\label{cor:proj-PSD}
	Provided that for any $\Mmat\in\Mset$ 
	we have $\Mmat_+\in\Mset$, then $\Omega_\Mset$ is convex. Moreover, $\proj_\Mset (G) \succeq 0$ for all $G\succeq 0$. 
\end{corollary}
Corollary \ref{cor:proj-PSD} establishes an important property about the iterates of the mirror-prox algorithm with $h(\cdot) = \frac{1}{2}\nor_2^2$ as the mirror map, corresponding to $\proxb{z}{\xi} = \proj_{\Zc}(z-\xi)$. 
If in Algorithm \ref{alg:mirror-prox} we initialize the part of $z_1$ corresponding to $\Mmat$'s to be a PSD matrix, all of such parts in the iterations $z_t$ and $w_t$ remain PSD as~1)~we add a PSD matrix ($\lambda X^TX$ from \eqref{eq:F-MP-vgf}) to the previous iteration, and,~2)~the projection onto $\Mset$ (which is not necessarily a subset of the PSD cone) ends up being a PSD matrix (by Corollary \ref{cor:proj-PSD}), hence it is equivalent to projection onto $\Mset\cap\Sbb_+$. Notice that such condition is required for applying the mirror-prox algorithm: the objective has to be convex-concave and the positive semidefiniteness of all iterations guarantees this property.

The above provides a glimpse into a more general approach in optimization with composite functions. While every convex function has a variational representation in terms of its conjugate function, namely $\Omega_\Mset(X) = \sup_Y~\langle X,Y \rangle  -  \Omega_{\Mset}^\star(Y)$, such expressions do not necessarily offer any computational advantage. 
With a more clever exploitation of the structure, 
$\Omega_\Mset(X)$ can be seen as a composition of the support function $S_\Mset(\cdot)$ with a structure mapping $g(X) = X^TX$, as in \eqref{eqn:supp-function}. Then, 
\begin{align*}
  \min_X ~~ \Loss(X) + \Omega_{\Mset}(X)
  &~~ \equiv ~~ 
  \adjustlimits\min_X \sup_Y ~~ \Loss(X) + \langle g(X),Y \rangle  -  S_{\Mset}^\star(Y)\\
  &~~ \equiv ~~ 
  \adjustlimits\min_X \sup_{Y\in \Mset} ~~ \Loss(X) + \langle X^TX,Y \rangle
\end{align*}
where we use the fact that $S_{\Mset}^\star(Y)$ is the indicator function for the set $\Mset$.  This can be seen as an interpretation 
for how our proposed algorithm replaces proximal mapping computations for $\Omega_\Mset$ with projections onto $\Mset$ (proximal mapping for the indicator function for $\Omega_\Mset$). Of course, to be able to use convex optimization algorithms, we will need to establish results similar to Lemma \ref{lem:proj-Z} and Corollary \ref{cor:proj-PSD}.

\newcommand{\opt}{\op{opt}}
\subsection{A Kernel Trick (Reduced Formulation)} \label{sec:reduced-form}
As we discussed earlier, when the loss function has 
the structure~\eqref{eq:loss-variational}, we can write  the VGF-penalized
minimization problem as a convex-concave saddle point problem
\begin{align}	\label{eq:saddle-omega}
J_{\text{opt}} 
= \adjustlimits\min_{X\in\R^{n\times m}} \max_{\g\in\cal{G}}\;\; 
 \iprod{X}{\cal{D}(\g)} - \Lhat(\g) +\lambda\,\Omega(X) \,.
\end{align}
Since $\cal{G}$ is compact, $\Omega$ is convex in $X\,$, and $\Lhat$ 
is convex in $\g\,$, we can use the minimax theorem 
to interchange the $\max$ and $\min$. Then, for any orthogonal matrix $Q$ we have
\begin{align}
J_\opt 
&= \adjustlimits\max_{\g\in\cal{G}} \min_X\;\;  \iprod{X}{\cal{D}(\g)} - \Lhat(\g) +\lambda\,\Omega(X)  \nonumber\\
&=\adjustlimits \max_{\g\in\cal{G}} \min_X\;\;  \iprod{Q^TX}{Q^T\cal{D}(\g)} - \Lhat(\g) +\lambda\,\Omega(Q^TX)  \nonumber\\
&=\adjustlimits \max_{\g\in\cal{G}} \min_X\;\;  \iprod{X}{Q^T\cal{D}(\g)} - \Lhat(\g) +\lambda\,\Omega(X)  \label{eq:saddle-Q-reduction}
\end{align}
where the second equality is due to the left unitarily invariance of $\Omega\,$, and we renamed the variable $X$ to get the third equality. Observe that $Q$ is an arbitrary orthogonal matrix in \eqref{eq:saddle-Q-reduction} and can be chosen in a clever way to simplify $\cal{D}$ as described in the sequel. 
Since $\cal{D}(\g)$ is linear in $\g\,$, consider a representation as
\begin{align}\label{eq:D-linear-kronecker}
\cal{D}(\g) = [D_1\g ~\cdots~ D_m\g] =  [D_1 ~\cdots~ D_m] (\eye_m \otimes \g) = \bd{D}(\eye_m \otimes \g), 
\end{align}
for some $D_i\in\R^{n\times p}$ and $\bd{D}\in\R^{n\times mp}\,$. Then, express $\bd{D}$ as the product of an orthogonal matrix and a residue matrix, such as in QR decomposition $\bd{D} = QR\,$, where provided that $n>mp\,$, only the first $mp$ rows of $R$ can be nonzero (will be denoted by $R_1$). Define $\cal{D}'(\g) = R_1(\eye_m\otimes \g)\in\R^{q\times m}$ for $q=\min\{mp,n\}\,$. 
Plugging the above choice of $Q$ in \eqref{eq:saddle-Q-reduction} gives 
\begin{align*}
J_\opt 
&= \adjustlimits\max_{\g\in\cal{G}} \min_{X_1,X_2}\;\; \iprod{ \begin{bmatrix}X_1 \\ X_2 \end{bmatrix}  }{ \begin{bmatrix}\cal{D}'(\g) \\ \zero \end{bmatrix} } - \Lhat(\g) +\lambda\,\Omega( \begin{bmatrix}X_1 \\ X_2 \end{bmatrix}  ) \,.
\end{align*}
Observe that setting $X_2$ to zero does not increase the value of $\Omega$ which allows for restricting the above to the subspace $X_2=0$ and getting
\begin{equation}	\label{eq:reduced-vgf}
J_\opt = \adjustlimits\min_{X\in\R^{q\times m}} \max_{\g\in\cal{G}}\; \iprod{X}{\cal{D}'(\g)} - \Lhat(\g) + \lambda\Omega(X)
\end{equation}
whose $X$ variable has $q=\min\{mp,n\}$ rows compared to $n$ rows in \eqref{eq:loss+omega}.

Notice that while the evaluation of $J_\opt$ via \eqref{eq:reduced-vgf} can potentially be more efficient, we are interested in recovering an {\em optimal point $X$} in \eqref{eq:saddle-omega} which is different from the optimal points in \eqref{eq:reduced-vgf}. 
Tracing back the steps we took 
from \eqref{eq:saddle-omega} to \eqref{eq:reduced-vgf}, we get
\begin{align*}
X_\opt^{\eqref{eq:saddle-omega}} = Q \begin{bmatrix} X_\opt^{\eqref{eq:reduced-vgf}} \\ \zero \end{bmatrix} \,.
\end{align*}

{
The special case of regularization with squared Euclidean norm has been understood and used before; e.g., see \cite{representer-2001}. 
However, the above derivations show that we can get similar results when the regularization can be represented as a maximum of squared weighted Euclidean norms. 
}

{
It is worth mentioning that the reduced formulation in~\eqref{eq:reduced-vgf} can be similarly derived via a dual approach; one has to take the dual of the loss-regularized optimization problem (e.g., see Example 11.41 in \cite{rockafellar2009variational}), use the left unitarily invariance of the conjugate VGF to reduce $\cal{D}$ to $\cal{D}'$, and dualize the problem again, to get~\eqref{eq:reduced-vgf}.
}

\subsection{A Representer Theorem}\label{sec:repr-thm}
A general loss-regularized optimization problem as in \eqref{eq:loss+omega} where the loss admits a Fenchel-type representation and the regularizer is a strongly convex VGF (including all squared vector norms) enjoys a representer theorem (see, e.g., \cite{representer-2001}). More specifically, the optimal solution is linearly related to the linear operator $\cal{D}$ in the representation of the loss. As mentioned before, for many common loss functions, $\cal{D}$ encodes the samples, which reduces the following proposition to the usual representer theorem. 

\begin{proposition}
For a loss-regularized minimization problem as in \eqref{eq:loss+omega}  
where $\Mset \subset \Sbb_{++}^m$ 
and $\Loss$ admits a Fenchel-type representation as 
\[
\Loss(X) = \max_{\g\in\cal{G}}\;\; \iprod{X}{\cal{D}(\g)} - \Lhat(\g) 
=  \max_{\g\in\cal{G}}\;\; \iprod{X}{ \bd{D}(\eye_m \otimes \g)  } - \Lhat(\g) \,,
\]
the optimal solution $X_\opt$ admits a representation of the form 
\[
X_\opt = \bd{D} C
\]
with a coefficient matrix $C$ given by $C = -\tfrac{1}{2\lambda}\Mmat_\opt^{-1} \otimes \g_\opt$ (optimal solutions of \eqref{eq:loss+omega-var}). 
\end{proposition}
\begin{proof}
Denote the optimal solution of \eqref{eq:loss+omega-var} by $(X_\opt,\g_\opt,\Mmat_\opt)\,$, which shares $(X_\opt,\g_\opt)$ with \eqref{eq:saddle-omega}. Consider the optimality condition as $-\tfrac{1}{\lambda}\cal{D}(\g_\opt)\in \partial \Omega(X_\opt)$ which implies 
$
X_\opt \in \partial \Omega^\st (-\tfrac{1}{\lambda}\cal{D}(\g_\opt))$. 
Now, suppose $\Mset \subset \Sbb_{++}^m$ which implies $\Omega_\Mset$ is strongly convex. Considering the characterization of subdifferential for $\Omega^\st$ from Proposition \ref{prop:subdiff-conj} as well as the representation of $\cal{D}(\g)$ in \eqref{eq:D-linear-kronecker} we get 
\begin{align*}
X_\opt =  -\tfrac{1}{2\lambda}\cal{D}(\g_\opt) \Mmat_\opt^{-1} 
= -\tfrac{1}{2\lambda} \bd{D}(\eye_m \otimes \g_\opt) \Mmat_\opt^{-1} 
= -\tfrac{1}{2\lambda} \bd{D}(\Mmat_\opt^{-1} \otimes \g_\opt)  \,.
\end{align*}
\end{proof}

This representer theorem allows us to apply our methods in more general reproducing kernel Hilbert spaces (RKHS) by choosing a problem specific reproducing kernel; e.g., see \cite{representer-2001, ZhouXiaoWu11}.  

\section{Numerical Example}	\label{sec:hier-class}
In this section, we discuss the application of VGFs in hierarchical classification to demonstrate the effectiveness of the presented algorithms in a real data experiment. More specifically, we compare the modified mirror-prox algorithm with adaptive line search presented in Section \ref{sec:mirror-prox} with the variant of Regularized Dual Averaging (RDA) method used in \cite{ZhouXiaoWu11} in the text categorization application discussed in \cite{ZhouXiaoWu11}.

\begin{figure}[t!]
\begin{subfigure}[b]{0.4\textwidth}
			\begin{center}
			\tikzset{
			  treenode/.style = {align=center, inner sep=0pt, text centered, font=\sffamily},
			  arn_n/.style = {treenode, circle, black, font=\sffamily\bfseries, draw=black, text width=1.3em},
			}
			\begin{tikzpicture}[scale=0.8, every node/.style={scale=0.8},
			  level/.style={sibling distance = 2cm/#1, 
			  level distance = 1.5cm}, 
			  level 1/.style={sibling distance=50pt}, 
			  level 2/.style={sibling distance=50pt}, 
			  edge from parent path={(\tikzparentnode) -- (\tikzchildnode)}] 
			\node [arn_n] (N0) {$0$}
			    child{ node [arn_n] (N1) {$1$} }
			    child{ node [arn_n] (N2) {$2$}
			            child{ node [arn_n] (N3) {$3$}}
			            child{ node [arn_n] (N4) {$4$}}
			    	edge from parent node [->,sloped]  {} 
					}
			; 
			\draw [thick, blue, ->, shorten >=.6em,shorten <=.8em, transform canvas={xshift=.3em,yshift=.3em}] 
				(N0.center) -- (N2.center) node [above, pos=.75, xshift=25pt] {$\x_1^T \ba < \x_2^T\ba$};
			\draw [thick, blue, ->, shorten >=.6em,shorten <=.8em, transform canvas={xshift=-.3em,yshift=.3em}] 
				(N2.center) -- (N3.center) node [above, pos=.75, xshift=-25pt] {$\x_3^T \ba < \x_4^T\ba$};
			\node [blue,right of=N0, xshift=-15pt] {$\ba$};
			\node [left of=N1, xshift=15pt] {$\x_1$};	
			\node [right of=N2, xshift=-13pt] {$\x_2$};	
			\node [left of=N3, xshift=15pt] {$\x_3$};	
			\node [left of=N4, xshift=15pt] {$\x_4$};	
			\node [blue, below of=N3, yshift=15pt] {$b=3$};	
			\end{tikzpicture}
	\end{center}
	\caption{}
	\label{fig:tree-classification}
\end{subfigure}%
\begin{subfigure}[b]{0.6\textwidth}
	\[
	f(\ba) = \left\{ \begin{array}{l}
	\textbf{initialize}~ i:=0 \\
	\textbf{while}~ \cal{C}(i)~ \mbox{is not empty} \\
	\qquad \displaystyle i:=\argmax_{j\in \cal{C}(i)} ~\x_j^T \ba \\
	\textbf{return}~ i
	\end{array} \right\}
	\]
	\caption{}
	\label{fig:hier-classifier}
\end{subfigure}
\vskip -.1in
\caption{\eqref{fig:tree-classification}: 
An example of hierarchical classification with four class labels
 $\{1,2,3,4\}$. The instance~$\bd{a}$ is classified recursively until 
 it reaches the leaf node~$b=3$, which is its predicted label.
 \eqref{fig:hier-classifier}: 
 Definition of the hierarchical classification function.
 }
\label{fig:hier-example}
\vskip -0.1in
\end{figure}
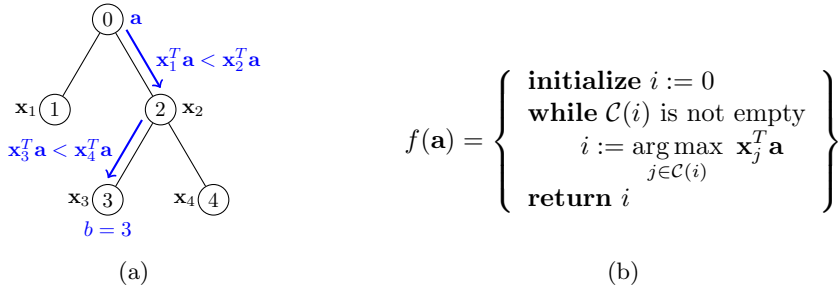

Let $(\ba_1, b_1), \ldots, (\ba_N, b_N)$ be a set of labeled data
where each $\ba_i\in\R^n$ is a feature vector and the associated
$b_i\in\{1,\ldots,m\}$ is a class label.
The goal of multi-class classification is to learn a classification 
function $f: \R^n \to \{1,\ldots,m\}$ so that, given any sample $\ba\in\R^n$ 
(not necessarily in the training set), the prediction $f(\ba)$ 
attains a small classification error compared with the true label.

In hierarchical classification, the class labels 
$\{1,\ldots,m\}$ are organized in a category tree, where the root 
of the tree is given the fictious label~$0$
(see Figure~\ref{fig:tree-classification}).
For each node $i\in\{0,1,\ldots,m\}$, 
let $\cal{C}(i)$ be the set of children of~$i$, 
$\cal{S}(i)$ be the set of siblings of~$i$,
and $\cal{A}(i)$ be the set of ancestors of~$i$
excluding~$0$ but including itself.
A hierarchical linear classifier $f(\ba)$ is defined in 
Figure~\ref{fig:hier-classifier}, which is parameterized
by the vectors $\x_1,\ldots,\x_m$ through a recursive procedure.
In other words, an instance is labeled sequentially by choosing the category 
for which the associated vector outputs the largest score among its siblings,
until a leaf node is reached.
An example of this recursive procedure is shown in 
Figure~\ref{fig:tree-classification}.
For the hierarchical classifier defined above, 
given an example $\ba_s$ with label~$b_s$, 
a correct prediction made by $f(\ba)$ implies 
that~\eqref{eq:multiclass-constraints} holds with
\[
  \cal{I}(k) = \bigl\{(i,j) \,:\, j\in\cal{S}(i), ~i\in\cal{A}(k)\bigr\} .
\]
Given a set of examples $(\ba_1, b_1), \ldots, (\ba_N, b_N)$, 
we can train a hierarchical classifier parametrized by 
$X=[\x_1,\ldots,\x_m]$ by solving the problem 
$\min_X \bigl\{\Loss(X) + \lambda\Omega(X)\bigr\}$,
with the loss function $\Loss(X)$ defined in~\eqref{eq:multiclass-hinge} and an appropriate VGF penalty function $\Omega(X)$.
As discussed in Section~\ref{sec:opt-algm}, the training optimization problem can be reformulated as a convex-concave saddle point problem of the form~\eqref{eq:loss+omega-var} and solved by the mirror-prox algorithm described in Section~\ref{sec:mirror-prox}. In addition, we can use the reduction procedure discussed in Section~\ref{sec:reduced-form} to reduce computational cost.

As discussed in \cite{ZhouXiaoWu11}, one can assume a model where classification at different levels of the hierarchy rely on different features or different combination of features. Therefore, authors in \cite{ZhouXiaoWu11} proposed regularization with $\abs{\x_i^T\x_j}$ whenever $j\in\mathcal{A}(i)\,$. A convex formulation of such a regularization function can be given in the form \eqref{eq:def-L1Gram} with 
\begin{align*}
\Mset = \bigl\{\Mmat: \;  \Mmat_{ii} = \Mbar_{ii}  \;,\; \abs{\Mmat_{ij}} = \abs{\Mbar_{ij}}   \bigr\}	
\end{align*}
where the nonzero pattern of $\Mbar$ corresponds to the pairs of ancestor-descendant nodes. According to \eqref{eq:eigen-Meff-L1}, we have $\Mset \subset \Sbb_+^m$ provided that $\lambda_{\min}(\Mcomp) \geq 0\,$; see Figure \ref{fig:boxVGF}. 

As a real-world example, 
we consider the classification dataset Reuters Corpus Volume I, 
RCV1-v2 \cite{lewis2004rcv1}, which is an archive of over 800,000 manually 
categorized newswire stories and is available in libSVM. 
A subset of the hierarchy of labels in RCV1-v2, with $m=23$ labels (18 leaves), is called ECAT and is used in our experiments. 
The samples and the classifiers are of dimension $n = 47236\,$. 
Lastly, there are 2196 training, and 69160 test samples available.

We solve the same loss-regularized problem as in \cite{ZhouXiaoWu11}, but using mirror-prox (discussed in Section \ref{sec:mirror-prox}) instead of regularized dual averaging (RDA). The regularization function is a VGF and is given in \eqref{eq:def-L1Gram}. A reformulation of the whole problem as a smooth convex-concave problem is given in \eqref{eq:hinge-divnorm-saddle}. To obtain comparable results, we use the same matrix $\Mbar$ and regularization parameter $\lambda=1$ as in \cite{ZhouXiaoWu11}. 
Note that in this experiment, $n=47236$ while $m=23$ and $p>2196$, so the kernel trick is not particularly useful since $n$ is not larger than $mp\,$.

{
Since we are solving the same problem as \cite{ZhouXiaoWu11}, the prediction error on test data 
will be the same as the error reported in this reference, which is better than the other methods. Moreover, one can look at the estimated classifiers and how well they validate the orthogonality assumption. Figure \ref{fig:heatmaps-MCAT} compares the pairwise inner products of classifiers estimated by our approach for hierarchical classification and those estimated by ``transfer'' method (see \cite{ZhouXiaoWu11} for details on this method).}
\newcommand\gray{gray}
\newcommand\ColCell[1]{%
  \pgfmathparse{mod(2*#1,2)?0:1}%
    \ifnum\pgfmathresult=0\relax\bf\leavevmode\color{red}\!\!\fi
  \pgfmathparse{int(round(#1))}
  	\edef\inp{\pgfmathresult}
  \pgfmathparse{1-.8*sqrt(abs(2*\inp/180-1))}%
  \expandafter\cellcolor\expandafter[\expandafter\gray\expandafter]\expandafter{\pgfmathresult}\inp\!}

\newcolumntype{P}[1]{>{\centering\arraybackslash}p{#1}}
\newcolumntype{E}{>{\collectcell\ColCell}P{15pt}<{\endcollectcell}}
\begin{figure}[ht]
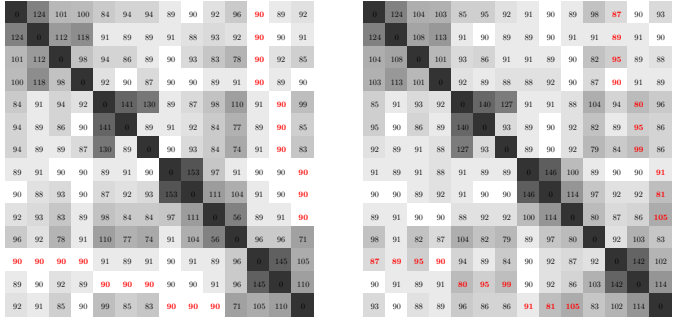

\begin{center}
\resizebox{.32\columnwidth}{!}
{
	\renewcommand{\arraystretch}{2.3}
	\noindent\begin{tabular}{*{25}{E}}
	0&124&101&100&84&94&94&89&90&92&96&90.1&89&92\\
	124&0&112&118&91&89&89&91&88&93&92&90.1&90&91\\
	101&112&0&98&94&86&89&90&93&83&78&90.1&92&85\\
	100&118&98&0&92&90&87&90&90&89&91&90.1&89&90\\
	84&91&94&92&0&141&130&89&87&98&110&91&90.1&99\\
	94&89&86&90&141&0&89&91&92&84&77&89&90.1&85\\
	94&89&89&87&130&89&0&90&93&84&74&91&90.1&83\\
	89&91&90&90&89&91&90&0&153&97&91&90&90&90.1\\
	90&88&93&90&87&92&93&153&0&111&104&91&90&90.1\\
	92&93&83&89&98&84&84&97&111&0&56&89&91&90.1\\
	96&92&78&91&110&77&74&91&104&56&0&96&96&71\\
	90.1&90.1&90.1&90.1&91&89&91&90&91&89&96&0&145&105\\
	89&90&92&89&90.1&90.1&90.1&90&90&91&96&145&0&110\\
	92&91&85&90&99&85&83&90.1&90.1&90.1&71&105&110&0\\
	\end{tabular}	
}
~~~
\resizebox{.32\columnwidth}{!}
{
	\renewcommand{\arraystretch}{2.3}
	\noindent\begin{tabular}{*{25}{E}}
	0&124&104&103&85&95&92&91&90&89&98&87.1&90&93\\
	124&0&108&113&91&90&89&89&90&91&91&89.1&91&90\\
	104&108&0&101&93&86&91&91&89&90&82&95.1&89&88\\
	103&113&101&0&92&89&88&88&92&90&87&90.1&91&89\\
	85&91&93&92&0&140&127&91&91&88&104&94&80.1&96\\
	95&90&86&89&140&0&93&89&90&92&82&89&95.1&86\\
	92&89&91&88&127&93&0&89&90&92&79&84&99.1&86\\
	91&89&91&88&91&89&89&0&146&100&89&90&90&91.1\\
	90&90&89&92&91&90&90&146&0&114&97&92&92&81.1\\
	89&91&90&90&88&92&92&100&114&0&80&87&86&105.1\\
	98&91&82&87&104&82&79&89&97&80&0&92&103&83\\
	87.1&89.1&95.1&90.1&94&89&84&90&92&87&92&0&142&102\\
	90&91&89&91&80.1&95.1&99.1&90&92&86&103&142&0&114\\
	93&90&88&89&96&86&86&91.1&81.1&105.1&83&102&114&0\\
	\end{tabular}	
}	
	
\end{center}
\caption{{Pairwise angles (in degrees) between the estimated classifiers for dataset MCAT (part of RCV1-v2 \cite{lewis2004rcv1}) via (left) regularization by the VGF in \eqref{eq:def-L1Gram} and (right) the ``transfer'' method (see \cite{ZhouXiaoWu11} and references therein). 
The boldface entries in red correspond to ancestor-descendant relations in the hierarchy of MCAT labels.}}
\label{fig:heatmaps-MCAT}
\end{figure}

In the setup of the mirror-prox algorithm, we use $\frac{1}{2}\nor_2^2$ as the mirror map which requires the least knowledge about the optimization problem (see \cite{JuditskyNemirovski11chapter6} for the requirements when combining a number of mirror maps corresponding to different constraint sets in the saddle point optimization problem). 
With this mirror map, the steps of mirror-prox only require orthogonal projection onto $\cal{G}$ and $\Mset\,$. 
The projection onto $\cal{G}$ in \eqref{eq:Gc-hinge} boils down to separate projections onto $N$ scaled 
simplexes (where the summation of entries is bounded by 1 and not necessarily equal to 1). Each projection amounts to zeroing out the negative entries followed by a projection onto the $\ell_1$ unit norm ball (e.g., using the simple process described in \cite{duchi2008efficient}).

The variant of RDA proposed in \cite{ZhouXiaoWu11} has a convergence rate of $O(\ln(t)/\sigma t)$ for the objective value, where $\sigma$ is the strong convexity parameter of the objective. On the other hand, mirror-prox enjoys a convergence rate of $O(1/t)$ as given in \cite{nemirovski2004prox}.  
Although there is a clear advantage to the MP method compared to RDA in terms of the theoretical guarantee, one should be aware of the difference between the notions of gap for the two methods. 
Figure \ref{fig:rate-compare} compares $\norm{X_t - X_{\text{final}}}_F$ for MP and RDA using each one's own final estimate $X_{\text{final}}\,$. 
In terms of the runtime, we empirically observe that each iteration of  MP takes about 3 times more time compared to RDA. However, as evident from Figure \ref{fig:rate-compare}, MP is still much faster 
in generating a fixed-accuracy solution. 
Figure \ref{fig:mp-rate} illustrates the decay in the value of the gap for mirror-prox method, $V_{z_t}(z_{t+1})\,$, which confirms the theoretical convergence rate of $O(1/t)$.
\newcommand{\subsamplerate}{10}
\pgfplotsset{
    convergenceplot/.style = {
        minor x tick num=1,
        xtick pos=left,
        ytick pos=left,
        enlarge x limits=false,
        every x tick/.style={color=black, thin},
        every y tick/.style={color=black, thin},
        tick align=outside,
        xlabel near ticks,
        ylabel near ticks,
        axis on top,        
	  	no markers, 
	  	very thick,
		legend style={
		                    draw=none, 
		                    text depth=0pt,
		                    anchor=north east,
		                    legend columns=-1,
		                    column sep=1cm,
		                    /tikz/column 2/.style={column sep=0pt},
		                    %
		                    /tikz/every odd column/.append style={column sep=0cm},
		                	},
    }
}
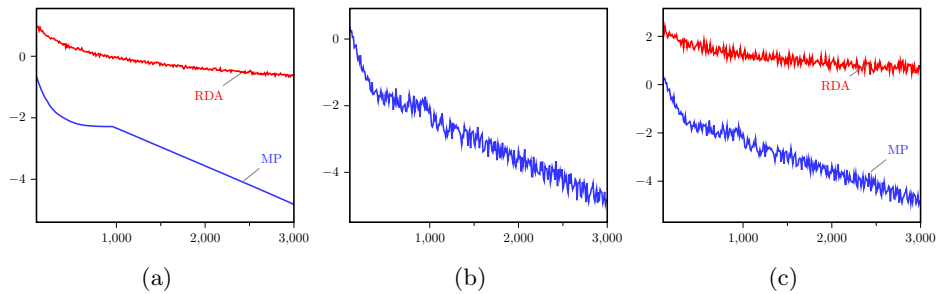
\begin{figure}[ht]
     \begin{subfigure}[b]{0.32\textwidth}
          \centering
          \resizebox{\linewidth}{!}{
\begin{tikzpicture}
	\begin{axis}[convergenceplot, 
		xmin = 100, xmax = 3000, xtick = {1000,2000,3000}, ytick = {2, 0,-2,-4},
		each nth point=\subsamplerate, filter discard warning=false, unbounded coords=discard]
	\addplot [color=blue!80] table [y=M, x=t]{FIG1.txt} node [pos=0.8,pin={45:MP},inner sep=0pt] {}; 
	\addplot table [y=R, x=t]{FIG1.txt} node [pos=0.8,pin={-135:RDA},inner sep=0pt] {}; 
	\end{axis}
\end{tikzpicture}          
}
          \caption{}
          \label{fig:rate-compare} 
     \end{subfigure}
     \begin{subfigure}[b]{0.32\textwidth}
          \centering
          \resizebox{\linewidth}{!}{
\begin{tikzpicture}
	\begin{axis}[convergenceplot, 
		xmin = 100, xmax = 3000, xtick = {1000,2000,3000}, ytick = {2, 0,-2,-4},
		each nth point=\subsamplerate, filter discard warning=false, unbounded coords=discard]
	\addplot [color=blue!80] table [y=M, x=t]{FIG2.txt};
	\end{axis}
\end{tikzpicture}      
}    
          \caption{}
          \label{fig:mp-rate} 
     \end{subfigure}
     \begin{subfigure}[b]{0.32\textwidth}
          \centering
          \resizebox{\linewidth}{!}{
\begin{tikzpicture}
	\begin{axis}[convergenceplot, 
		xmin = 100, xmax = 3000, xtick = {1000,2000,3000}, ytick = {2, 0,-2,-4},
		each nth point=\subsamplerate, filter discard warning=false, unbounded coords=discard]
	\addplot [color=blue!80] table [y=M, x=t]{FIG3.txt} node [pos=0.8,pin={45:MP},inner sep=0pt] {}; 
	\addplot table [y=R, x=t]{FIG3.txt} node [pos=0.8,pin={-135:RDA},inner sep=0pt] {}; 
	\end{axis}
\end{tikzpicture}          
}
          \caption{}
          \label{fig:loss-compare} 
     \end{subfigure} 
     \vskip -.17in
    \caption{Convergence behavior for mirror-prox and RDA in our numerical experiment. 
        (\protect\subref{fig:rate-compare}) Average error over the $m$ classifiers between each iteration and the final estimate, $\norm{X_t - X_{\text{final}}}_F\,$. 
        (\protect\subref{fig:mp-rate}) MP's gap $V_{z_t}(z_{t+1})$. 
        (\protect\subref{fig:loss-compare}) The value of loss function relative to the final value. 
        For visualization purposes, all of the plots show data points at every 10 iterations. All vertical axes have a logarithmic scale. 
    }
    \label{fig:convergence-rates}      
 \end{figure}

\section{Discussion}
In this paper, we introduce variational Gram functions, which include many existing regularization functions as well as important new ones. Convexity properties of this class, conjugate functions, subdifferentials, semidefinite representability, proximal operators, and other convex analysis properties are studied. 
By exploiting the structure in loss and the regularizer, namely $\Loss(X) = \Lhat^\st(\cal{D}^\st(X))$ and $\Omega_\Mset(X) = S_\Mset(X^TX)$, we provide various tools and insight into such regularized loss minimization problems: By adapting the mirror-prox method \cite{nemirovski2004prox}, we provide a general and efficient optimization algorithm for VGF-regularized loss minimization problems. We establish a general kernel trick and a representer theorem for such problems. Finally, the effectiveness of VGF regularization as well as the efficiency of our optimization approach is illustrated by a numerical example on hierarchical classification for text categorization.

There are numerous directions for future research on this class of functions. 
One issue to address is how to systematically pick an appropriate set $\Mset$ when defining a new VGF for some new application. 
Statistical properties of VGFs, for example the corresponding sample complexity, are of interest from a learning theory perspective. 
The presented kernel trick 
(which uses the left unitarily invariance property of VGFs) 
can be potentially extended to other invariant regularizers. 
And last but not least, it is interesting to see if there is a variational Gram representation for any squared left unitarily invariant norm.

\bibliographystyle{siam}
\bibliography{VGFbib}

\appendix
\section{Additional Proofs} 
\begin{proof}[Proof of Lemma \ref{lem:homog-norm}]
First, assume that $\Omega$ is convex. By plugging in $X$ and $-X$ in the definition of convexity for $\Omega$ we get $\Omega(X)\geq 0\,$, so the square root is well-defined. 
We show the triangle inequality $\sqrt{\Omega(X+Y)}\leq \sqrt{\Omega(X)} + \sqrt{\Omega(Y)}\,$ holds for any $X,Y$.  
If $\Omega(X+Y)$ is zero, the inequality is trivial. 
Otherwise, for any $\theta\in(0,1)$ let $A = \tfrac{1}{\theta}X$, $B= \tfrac{1}{1-\theta}Y$, and use the
convexity and second-order homogeneity of $\Omega$ to get
\begin{align}	\label{eq:quad-cvx}
\Omega(X+Y) = \Omega(\theta A + (1-\theta)B) \leq \theta \Omega(A) + (1-\theta )\Omega(B) = 
\tfrac{1}{\theta}\Omega(X)+ \tfrac{1}{1-\theta}\Omega(Y).
\end{align}
If $\Omega(X) \geq \Omega(Y)=0\,$, set $\theta = (\Omega(X)+\Omega(X+Y))/(2\Omega(X+Y)) >0$. 
Notice that $\theta\geq 1$ provides $\Omega(X) \geq \Omega(X+Y)$ as desired. On the other hand, if $\theta<1\,$, we can use it in \eqref{eq:quad-cvx} to get the desired result as 
\[
\Omega(X+Y) \leq \tfrac{1}{\theta}\,\Omega(X) 
= \frac{2\Omega(X+Y)\Omega(X)}{\Omega(X+Y)+\Omega(X)}
\implies 
\Omega(X) \geq \Omega(X+Y) \,.
\]
And if $\Omega(X),\Omega(Y) \neq 0\,$, set $\theta = \sqrt{\Omega(X)}/(\sqrt{\Omega(X)}+\sqrt{\Omega(Y)}) \in (0,1)$ to get
\[
\Omega(X+Y) \leq \tfrac{1}{\theta}\,\Omega(X)+ \tfrac{1}{1-\theta}\,\Omega(Y)
= (\sqrt{\Omega(X)}+\sqrt{\Omega(Y)})^2 \,.
\]
Since $\sqrt{\Omega}$ satisfies the triangle inequality and absolute homogeneity, it is a semi-norm. Notice that 
$\Omega(X)=0$ does not necessarily imply $X=0$, unless $\Omega$ is strictly convex.  

Now, suppose that $\sqrt{\Omega}$ is a semi-norm; hence convex. 
The function $f$ defined by $f(x)=x^2$ for $x\geq 0$ and $f(x)=0$ for $x\leq 0$ is non-decreasing, so the 
composition of these two functions is convex and equal to $\Omega\,$. 
It is worth mentioning that one can alternatively use Corollary 15.3.2 of \cite{Rockafellar70book} to prove the first part of the lemma. 
\end{proof}

\begin{lemma}\label{lem:Psi-convex-via-norms}
Consider any norm $\nor\,$. Then, $\norm{\abs{\cdot}}$ is a norm itself if and only if we have $\norm{\abs{x}} = \min_{y\geq\abs{x}}\; \norm{y}\,$. 
\end{lemma}
\begin{proof}[Proof of Lemma \ref{lem:Psi-convex-via-norms}]
First, suppose $\nor_a := \norm{\abs{\cdot}}$ is a norm; hence it is an absolute norm and is monotonic as well by definition. Therefore, for any $y\geq \abs{x}$ we have $\norm{y}_a \geq \norm{x}_a$ which gives $\min_{y\geq\abs{x}}\; \norm{y}_a \geq \norm{x}_a\,$. Since $\abs{x}$ is feasible in this optimization, and $\norm{\abs{x}}_a = \norm{x}_a$ we get the desired result; $\norm{\abs{x}} = \norm{x}_a =\min_{y\geq\abs{x}}\; \norm{y}\,$.

On the other hand, consider $f(\cdot) := \min_{y\geq\abs{x}} \norm{y}$. We show that it is a norm. Clearly, $f$ is nonnegative and homogenous, and $f(x)=0$ implies that $\norm{y}=0$ for some $y\geq \abs{x}\geq 0$ which implies $x=0\,$. The triangle inequality can be verified as, 
\begin{align*}
f(x+z) 
= \min_{y\geq \abs{x+z}} \; \norm{y}
\leq \min_{y\geq \abs{x}+\abs{z}} \; \norm{y} 
= \min_{y_1\geq \abs{x} \;,\; y_2\geq \abs{z}} \; \norm{y_1+y_2}   \qquad\qquad\\
\leq \min_{y_1\geq \abs{x} \;,\; y_2\geq \abs{z}} \; \norm{y_1}+\norm{y_2}
= f(x)+f(z) \,.
\end{align*}
\end{proof}

\end{document}